\documentclass[12pt,oneside]{amsart}
\usepackage{amssymb, amsmath, amsthm}

\usepackage{epsfig}
\usepackage{graphicx}
\usepackage{color}
\usepackage{enumerate}
\usepackage{amsrefs}

\usepackage[normalem]{ulem}

\usepackage[normalem]{ulem}

\theoremstyle{plain}
\newtheorem{theorem}{Theorem}[section]

\newtheorem*{theorem*}{Theorem}
\newtheorem*{theorem-DisjtSS}{Theorem \ref{Thm: Disjt SS}}
\newtheorem*{theorem-EssentialTorus}{Theorem \ref{Thm: Essential Torus}}
\newtheorem*{cor-ScharlWu}{Corollary \ref{Cor-ScharlWu}}
\newtheorem*{corollary-MSC 2}{Corollary \ref{Cor: MSC 2}}
\newtheorem*{theorem-Main A}{Theorem \ref{Thm: Main A}}
\newtheorem*{theorem-Main B}{Theorem \ref{Thm: Main Thm B}}
\newtheorem*{Cor-Unknotting}{Theorem \ref{Cor: Prime Unknotting 1}}
\newtheorem*{Cor-genusbandsum}{Theorem \ref{Thm: Genus superadd}}
\newtheorem*{Cor-bandsumscc}{Corollary \ref{Cor: Band Sums CC}}

\newtheorem{proposition}[theorem]{Proposition}
\newtheorem{corollary}[theorem]{Corollary}
\newtheorem{lemma}[theorem]{Lemma}
\newtheorem{definition}[theorem]{Definition}
\newtheorem{remark}[theorem]{Remark}

\theoremstyle{definition}

\newtheorem*{question}{Question}

\newcommand{\R}{{\mathbb R}}

\newcommand{\Z}{\mathbb Z}

\newcommand{\LL}{\mathcal{L}}

\newcommand{\bdd}{\partial}

\newcommand{\ra}{\rightarrow}
\newcommand{\x}{\times}

      \makeatletter
      \def\@setcopyright{}
      \def\serieslogo@{}
      \makeatother

\begin{document}

   \title[Prime decomposition of 2-string links]{A prime decomposition theorem for the 2-string link monoid}
   \author{Ryan Blair, John Burke, Robin Koytcheff}
   \email{}
   \thanks{}

\begin{abstract}
In this paper we use 3-manifold techniques to illuminate the structure of the string link monoid. In particular, we give a prime decomposition theorem for string links on two components as well as give necessary conditions for string links to commute under the stacking operation.
\end{abstract}

\maketitle
\date{\today}

\section{Introduction}
It is well known that isotopy classes of knots form a monoid via the operation of connected sum.  This operation is well-defined on both closed knots (embeddings of $S^1$ into $S^3$ or $\R^3$) and long knots (embeddings of $\R$ into $\R^3$ which agree with a fixed linear embedding outside of a compact set).  However, if one tries to generalize this operation to links of more than one component, one finds an analogous monoid structure not for closed links, but for string links.  This operation is given by ``stacking", and the resulting monoid structure is the subject of this paper.

The prime decomposition theorem for knots proven by Schubert in 1949 states that the monoid of isotopy classes of knots  is the free commutative monoid on the isotopy classes of prime knots~\cite{S49}.  One main result of the current paper is an analogous theorem for 2-component string links. An analogue of Schubert's theorem is well known for closed links of any number of components under the operation connected-sum (in the sense of \cite[p.~40]{R76}, i.e., ``along one component"), which is quite different from the operation of stacking.  Our main result is the following:
\begin{theorem*}[Corollary \ref{cor:main}]
A
2-component string link $L$ can be written as the product (under stacking) of prime factors $L=L_1\# ... \# L_n$, where this decomposition is unique up to permuting the order of central elements and multiplication by units.
\end{theorem*}
In the 2-string link monoid, the center turns out to be generated by \emph{split string links} and \emph{one-strand cables}, while the units are braids.  Along the way, we consider the $n$-string link monoid for any $n$, establishing the necessary conditions for commutativity and characterizing the units and the center of the monoid of $n$-component string links .

This work was motivated by work of the last two authors on operad actions on spaces of string links~\cite{BK}. In turn, that work built upon papers of Budney~\cite{B1, B2}, which, roughly speaking, generalize decomposition theorems about the monoid of isotopy classes of long knots (or $\pi_0$ of the space of long knots) to the level of the \emph{whole space} of long knots.  In seeking generalizations from long knots to string links, it was both natural and necessary to understand the structure of the monoid of isotopy classes of string links.  The results in the current paper regarding 2-string links allows the last two authors to prove a decomposition theorem for a large part of the \emph{whole space} of 2-string links.

\subsection{Acknowledgments}
We thank the referee for a careful reading of the paper, insightful comments, and substantive suggestions.  The last author was partially supported by NSF grant DMS-1006410.  He thanks Tom Goodwillie for useful conversations.

\section{Basic Notions}

An \emph{$n$-string link} is a properly, smoothly embedded collection of $n$ arcs $T$ in $M=D^2 \x I$ such that each arc has one endpoint in $\partial_- M = D^2\times \{0\}$ and one endpoint in $\partial_+ M = D^2\times \{1\}$ and each of $\partial_- M \cap T$ and $\partial_+ M \cap T$ are a collection of $n$ points $\{(x_1,0),(x_2,0),...,(x_n,0)\}$ and $\{(x_1,1),(x_2,1),...,(x_n,1)\}$ respectively.
Furthermore, we require that at each pair of endpoints, the derivatives of all orders of the $n$ embeddings agree with the maps $t \mapsto (x_i, t)$.
An $n$-string link is \emph{pure} if for every component $\alpha$ of $T$, $\alpha\cap \partial_- M=(x_i,0)$ implies $\alpha\cap \partial_+ M=(x_i,1)$.

\begin{definition}[String link equivalence]
\label{StringLinkEquivalence}
String links $T_1$ and $T_2$ in $M$ are equivalent if there is an isotopy of $M$ fixing $\partial M$ that takes $T_1$ to $T_2$.  In this case, (by abuse of notation) we will write $T_1 = T_2$.
\end{definition}

Let $h:D^2\times I \ra I$ be the projection map onto the second coordinate. A \emph{braid} is a string link that is equivalent to a string link $T$ with the property that for every component $\alpha$ of $T$, the restriction of $h$ to $\alpha$ is a smooth, one-to-one and onto function with no critical points in its domain. Given $n$-string links $T_1$ in $M_1$ and $T_2$ in $M_2$, denote the \emph{stacking}
operation by $T_1\# T_2$ which is achieved by gluing $\partial_+ M_1$ to $\partial_- M_2$ via the identity map and considering the $n$-string link which is the image of $T_1 \cup T_2$ under this identification. The resulting quotient of $M_1 \cup M_2$ is again homeomorphic to $D^2\times I$ and we choose to identify the image of $M_1$ with $D^2 \x [0,1/2]$ and the image of $M_2$ with $D^2 \x [1/2,1]$ in the obvious ways. See Figure \ref{StringlinkStacking}.

\begin{figure}[h!]
\begin{picture}(315,270)
\put(0,11){\includegraphics[scale=.6]{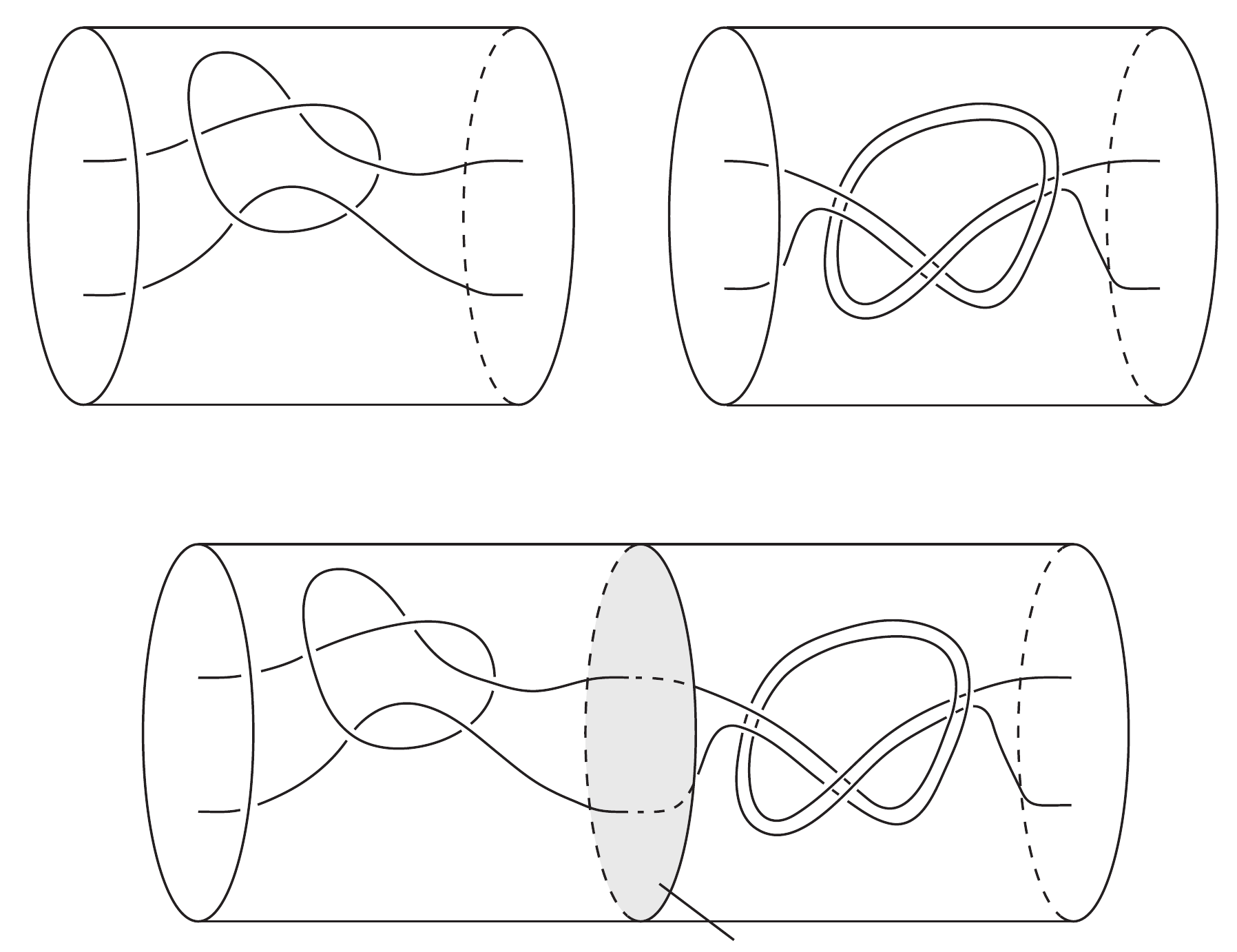}}
\put(73,137){$T_1$}
\put(233, 137){$T_2$}
\put(140, 0){$T_1 \# T_2$}
\put(183,7){$F$}
\put(-5,139){$\partial_- M$}
\put(130,140){$\partial_+ M$}
\end{picture}
\caption{The stacking of two pure string links}
\label{StringlinkStacking}
\end{figure}
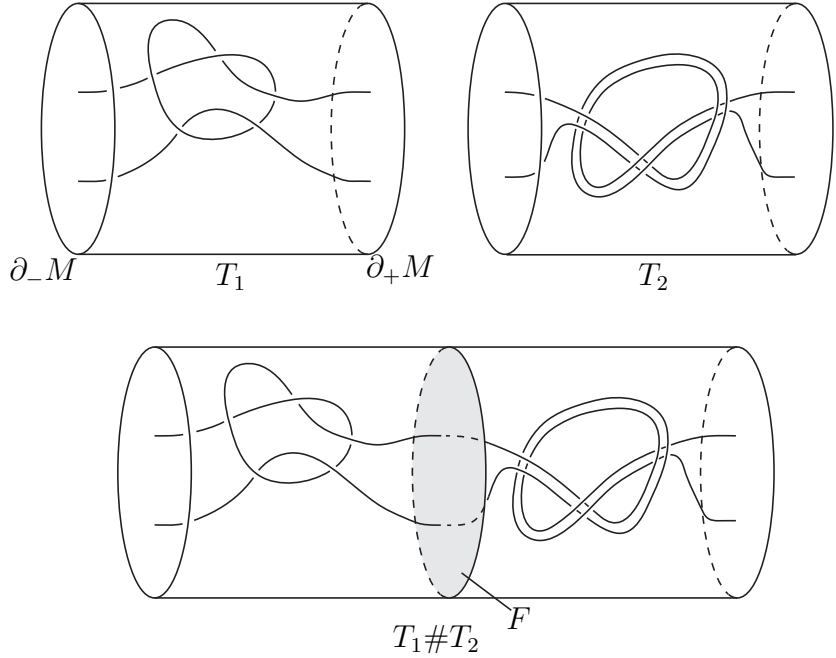

\begin{definition}
\label{k-punctured}
Given a compact 1-manifold $T$ properly embedded in a compact 3-manifold $M$, an embedded surface $F$ in $M$ is \emph{$k$-punctured} if $F$ meets $T$ transversely in exactly $k$ points. For example, $F$ in Figure \ref{StringlinkStacking} is 2-punctured.
\end{definition}

\begin{definition}
\label{ProperIsotopy}
Given a surface $F$ properly embedded in a 3-manifold $M$ and a 1-manifold $T$ properly embedded in $M$, an isotopy $\phi_t$ of $F$ in $M$ is \emph{proper} if it extends to an isotopy of pairs $\phi_t:(F,\partial F)\ra (M,\partial M)$. An isotopy $\phi_t$ of $F$ in $M$ is \emph{transverse} to $T$ if the image of $\phi_t$ is transverse to $T$ for every fixed value of $t$.
\end{definition}

\begin{remark}
\label{TransverseT}
 Unless explicitly stated otherwise, all isotopies of surfaces in this paper will be proper isotopies that are transverse to the string link under consideration.\end{remark}

\begin{definition}
\label{BoundaryParallel}
Given a string link $T$ in $M$, a $k$-punctured, properly embedded surface $F$ in $M$ (as in Definition \ref{k-punctured}) is called \emph{boundary-parallel} if there is an isotopy of $F$ in $M$ which takes $F$ to a subsurface contained in $\partial M$.
\end{definition}

\begin{definition}
A \emph{decomposing disk} for an $n$-string link $T$ in $M$ is an $n$-punctured disk which is properly embedded in $M$,
whose boundary is isotopic in $\partial(D^2 \x I)$ to $\partial (\partial_+ M)$, and which is not boundary-parallel (i.e. not isotopic to $\bdd_+ M$ nor $\bdd_- M$). See the disk $F$ in Figure \ref{StringlinkStacking}.
\end{definition}

Note that each strand of a string link passes through a decomposing disk exactly once.

\textbf{Further notation and terminology:}  A decomposing disk $F$ for a string link $T$ in $M$ separates
$M$ into two pieces, one containing $\partial_- M$ and the other containing $\partial_+ M$.  Each piece is homeomorphic to $D^2 \x I$, so the two pieces that $F$ separates $T$ into can be identified with two string links, up to multiplication by braids (cf.~\emph{braid-equivalence} in Definition \ref{BraidEquivalenceDef} and its alternative formulation in Proposition \ref{AltStringLinkEquiv}).  If $K$ is the string link resulting from restricting $T$ to the piece of $M$ containing $\partial_-M$ and $L$ is the string link resulting from restricting $T$ to the piece containing $\partial_+ M$, we say that $F$ \emph{decomposes $T$ as $K\# L$}.
Note that the order of $K$ and $L$ matters.
 We will let $M^F_K$ and $M^F_L$ denote the closures of the complementary components of $F$ in $M$ that contain $K$ and $L$ respectively.

\begin{definition}
A loop $\gamma$ embedded in a punctured surface $F$ is \emph{essential} if it does not bound a 0-punctured disk or a 1-punctured disk in $F$.
\end{definition}

\begin{definition}
Given a string link $T$ in $M$ and a properly embedded, possibly punctured surface $F$ in $M$, $F$ is \emph{compressible} if
there exists a disk $D$ embedded in $M$ such that $D \cap F=\partial D$, $\partial D$ is essential in $F$ and $D$ is disjoint from $T$. Such a disk is called a \emph{compressing disk}.  Otherwise, $F$ is \emph{incompressible}. A surface $F$ is \emph{essential} if $F$ is incompressible and non-boundary parallel.
\end{definition}

Here we take the point of view of studying surfaces $F$ properly embedded in $M$ which may be punctured by the string link $T$.
An alternative perspective would be to study $F\cap M^*$, where $M^*$ is the complement of an open regular neighborhood of $T$ in $M$.  From that perspective, $F$ is \emph{incompressible} in $M$ if the inclusion $F\cap M^*\hookrightarrow M^*$ induces an injection on $\pi_1$.

\begin{lemma}\label{incompdecomp}
Any decomposing disk $F$ for a string link $T=K\# L$ in $M$ is incompressible.
\end{lemma}

\begin{proof}
Suppose $F$ is a decomposing disk.  Then $F$ separates $M$ into $M^F_K$ and $M^F_L$.  Suppose $F$ is compressible.  Let $D$ be a compressing disk for $F$.  We can assume $D$ is completely contained in $M^F_K$ or $M^F_L$. Without loss of generality, assume $D$ is contained in $M^F_K$. Since $D$ is a compressing disk, then, by definition, $\partial D$ is essential in $F$. Since $F$ is a punctured disk, this implies $\partial D$ bounds a punctured disk $D_1\subset F$. Since the 3-ball is irreducible, $D_1 \cup D$ bounds a 3-ball in $M_K^F$. Hence, any arc of $T$ in the 3-ball bounded by $D_1\cup D$ is an arc in $K$ which must have \emph{both} endpoints in $D_1 \subset \partial_+ M^F_K (= \partial_- M^F_L)$.  This contradicts the fact that each strand of a string link passes through a decomposing disk only once.
\end{proof}

The following proposition was originally proven in \cite{K}. Here we provide an alternative proof.

\begin{proposition}
\label{BraidsProp}
If  $K$ and $L$ are $n$-string links such that $K\# L$ is a braid, then both $K$ and $L$ are braids.
\end{proposition}

\begin{proof}
Let $F$ be the $n$-punctured disk corresponding to $\partial_+ M^F_K$ and $\partial_- M^F_L$ in $M$.
By Lemma \ref{incompdecomp}, $F$ is incompressible. Since $K\# L$ is a braid, $M$ can be decomposed as an $I$-bundle over $D^2$ with $K\# L$ the union of $I$-fibers. Since $F$ is incompressible and separates $\partial_- M$ and $\partial_+ M$, then, by Corollary 3.2 of \cite{W68}, $F$ is isotopic to both $\partial_+ M$ and $\partial_- M$. By isotopy extension, $M^F_K$ can be decomposed as an $I$-bundle over $D^2$ with $K$ the union of $I$-fibers and $M^F_L$ can be decomposed as an $I$-bundle over $D^2$ with $L$ the union of $I$-fibers. Thus, both $K$ and $L$ are braids.
\end{proof}

\begin{corollary}
The units in the $n$-string link monoid are precisely the braids.
\end{corollary}
\begin{proof}
The trivial string link is a braid, so this follows immediately from Proposition \ref{BraidsProp}.
\end{proof}

\begin{definition}
(a)  An $n$-string link $L$ embedded in $M=D^2 \times I$ is \emph{prime} if $L=J\# K$ implies $J$ or $K$ is a braid and if $L$ is not a braid itself.

(b) Equivalently, $L$ is \emph{prime} if every $n$-punctured disk
properly embedded in $M$ with boundary isotopic to $\partial (\partial_+ M)$ is boundary parallel, but $\partial_+ M$ is not isotopic to $\partial_- M$.

(c) Equivalently, a string link is \emph{prime} if it is not a braid and has no decomposing disks.
\end{definition}

We leave it to the reader to verify the equivalence of the above three definitions.


\section{Commutativity in the String Link Monoid}

In this Section, we study when two string links commute.  The most difficult result is establishing some necessary conditions for commutativity, namely Proposition \ref{commute}.  We also establish some sufficient conditions in Proposition \ref{commutecon}.

Recall that, given a string link $T$, we only consider isotopies of surfaces embedded in $M$ that are everywhere transverse to $T$.

\begin{definition}
\label{SplitLink}
An $n$-string link $T$ embedded in $M=D^2 \times I$ is \emph{split} if there exists a compressing disk for $\partial M$ which meets $\partial (\partial_- M)$ minimally in exactly two points.
\end{definition}

Note that any braid on two or more strands is a split link.

\begin{figure}[h!]
\begin{picture}(150,100)
\put(0,0){\includegraphics[scale=.6]{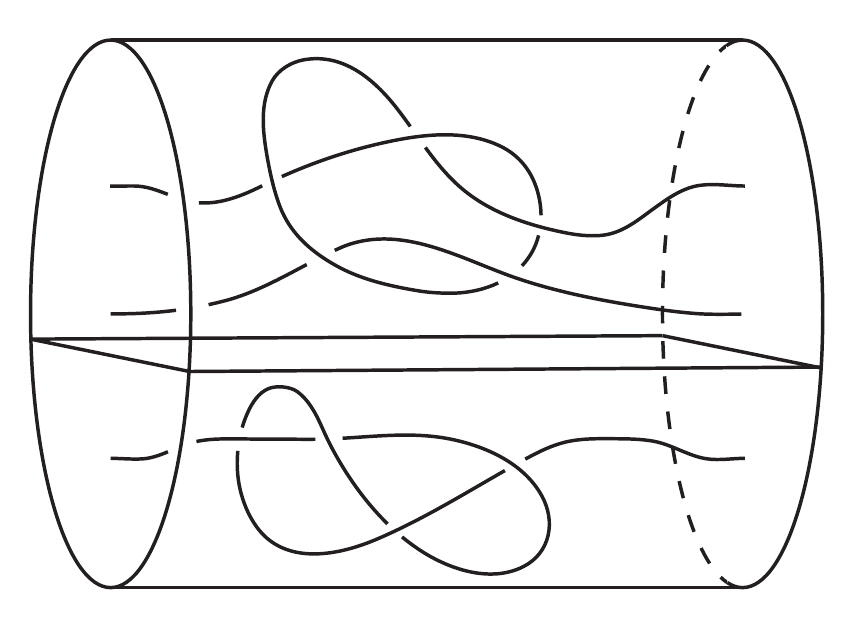}}
\end{picture}
\caption{A split link. The rectangular disk shown is the compressing disk in Definition \ref{SplitLink}.}
\label{Splitdisk}
\end{figure}

\begin{definition}
A \emph{splitting disk} for an $n$-string link $T$ in $M=D^2 \x I$ is a properly embedded punctured disk $F$ in $M$ such that $\partial F$ is contained in $\partial_- M$ (resp. $\partial_+ M$), each strand of $T$ meets the splitting disk at most once and the splitting disk is not transversely isotopic to a subdisk of $\partial_- M$ (resp. $\partial_+ M$).
\end{definition}

\begin{figure}[h!]
\begin{picture}(250,100)
\put(0,10){\includegraphics[scale=.6]{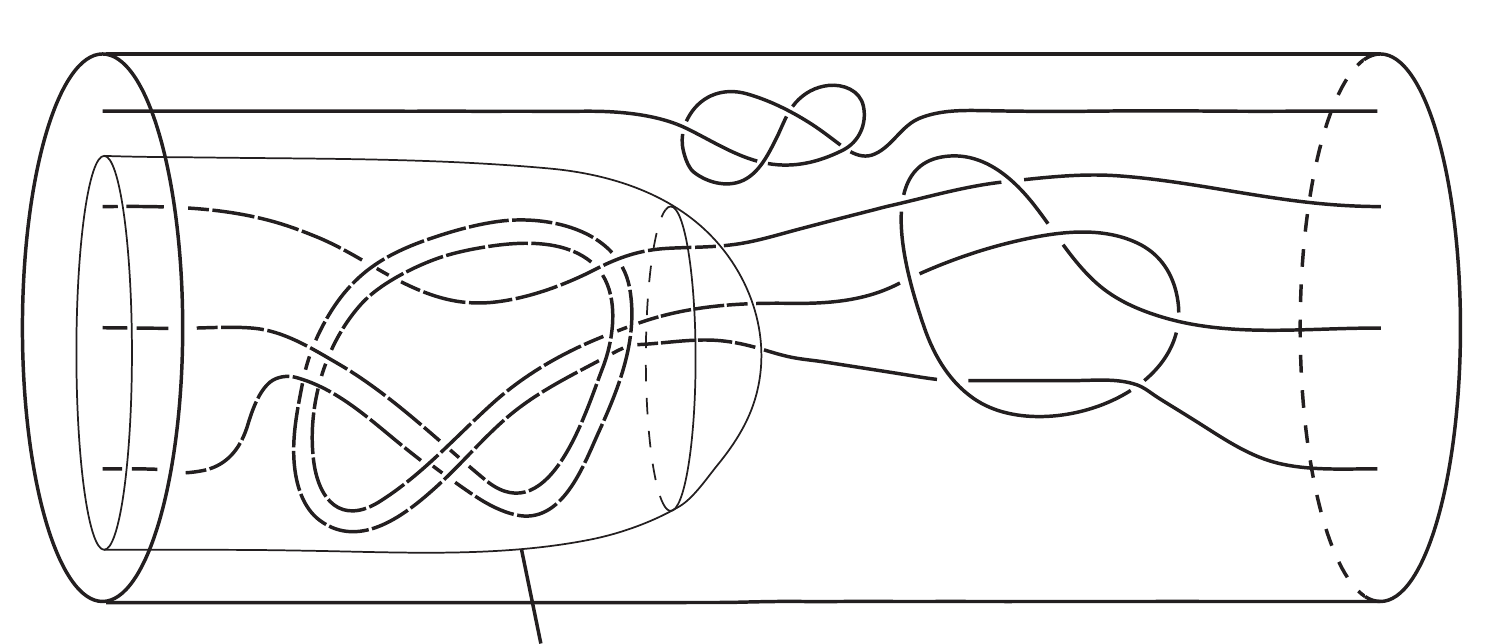}}
\put(90,0){$F$}
\end{picture}
\caption{A splitting disk.}
\label{Splitdisk}
\end{figure}

\begin{lemma}\label{splitdiskprime}
Let $T$ in $M$ be a prime string link with splitting disk $F$ such that $\partial F \subset \partial_{\pm} M$.  Then the union of the punctured annulus in $\partial_{\pm} M$ cobounded by $\partial \partial_{\pm} M$ and $\partial F$ together with the punctured disk $F$ is isotopic to $\partial_{\mp} M$.
\end{lemma}

\begin{figure}[h!]
\begin{picture}(250,105)
\put(30,0){\includegraphics[scale=.6]{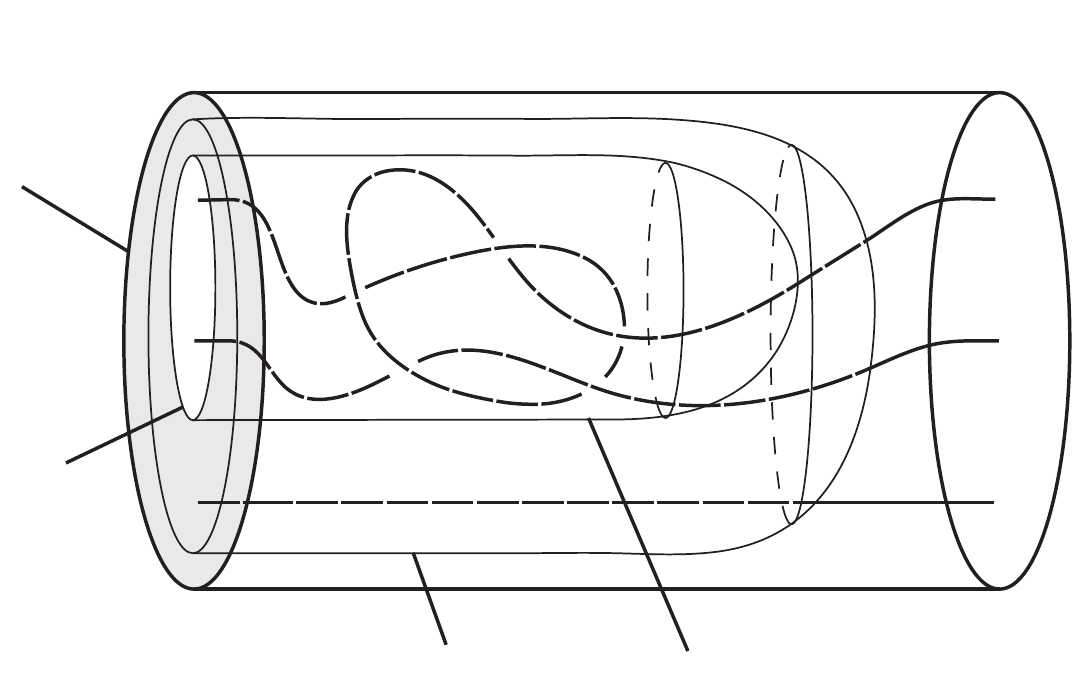}}
\put(10,90){$\partial \partial_{-} M$}
\put(26,31){$\partial F$}
\put(147,-2){$F$}
\put (105,-1){$F'$}
\end{picture}
\caption{A prime string link and splitting disks $F$ and $F'$.}
\label{splitdisks}
\end{figure}
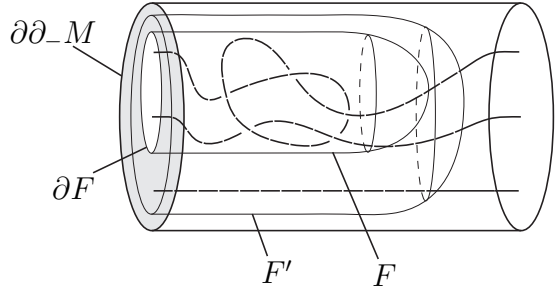

\begin{proof}
The union of the punctured annulus in $\partial_{\pm} M$ cobounded by $\partial \partial_{\pm} M$ and $\partial F$ together with the punctured disk $F$ is a punctured disk $D$ that meets each component of $T$ exactly once and has boundary $\partial \partial_{\pm} M$. After a small transverse isotopy of $D$ that pushes $D$ off of $\partial_{\pm} M$, $D$ is a properly embedded disk in $M$. Since $T$ is prime, $D$ must be isotopic to either $\partial_{-} M$ or $\partial_{+} M$. From one of these two possibilities we conclude that $F$ is isotopic to the subdisk of $\partial_{\pm} M$ with boundary $\partial F$, a contradiction to the definition of splitting disk. The other possibility is the desired conclusion.

\end{proof}

Note that the number of punctures in a splitting disk is not required to be minimal. Thus a prime, split string link may have non-isotopic splitting disks, each intersecting a different number of ``trivial strands." See Figure \ref{splitdisks}. However, a prime non-split string link has a unique splitting disk $F$ up to isotopy, provided we fix which side of $\partial_M$ ($\partial_-$ or $\partial_+$) its boundary $\partial F$ lies in.

\begin{definition}
\label{CablingAnnulusDefn}
If $T$ embedded in $M$ is a string link, then a properly embedded annulus $A$ in $M$ is a \emph{cabling annulus} if one component of $\partial A$ is contained in $\partial_- M$, the other component is contained in $\partial_+ M$ and, if $N$ is the copy of $D^2 \times I$ that $A$ cobounds with two disks in $\partial M$, then $T\cap N \neq \emptyset$ and $T\cap N$ is a braid in $N$.  (Note that we do not consider the empty submanifold to be a braid.)
\end{definition}

\begin{figure}[h!]
\begin{picture}(175,122)
\put(0,5){\includegraphics[scale=.7]{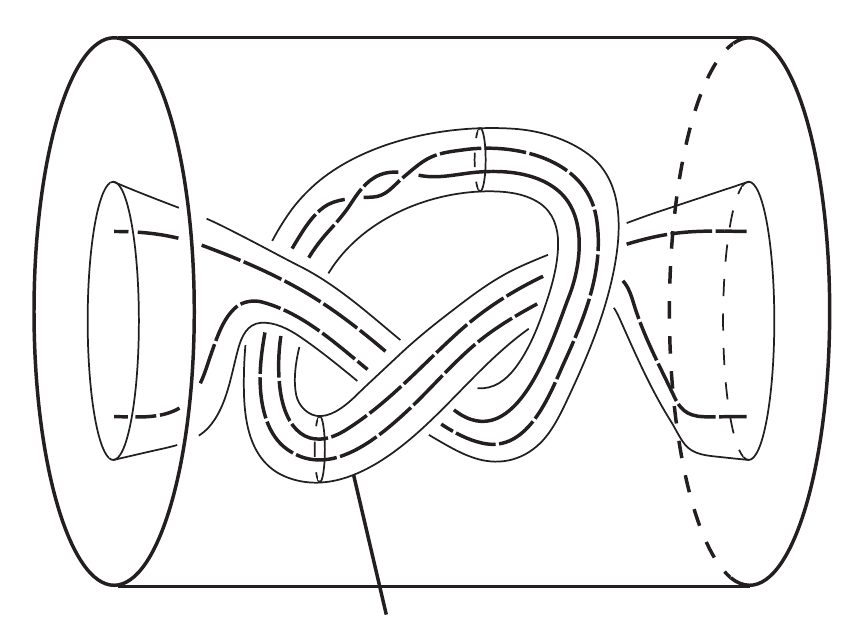}}
\put(75,0){$A$}
\end{picture}
\caption{Cabling}
\label{Cabling Annulus}
\end{figure}

Note that the boundary of a regular neighborhood of a single strand of a string link contains a cabling annulus. \\

An isotopy of a stacking of two string links will not necessarily fix $D^2 \x \{\frac{1}{2}\}$.
This motivates the following definition.

\begin{definition}
\label{BraidEquivalenceDef}
Two n-string links $T_1$ and $T_2$ are \emph{braid-equivalent} if there exist braids $B_1$ and $B_2$ such that $T_1=B_{1}\# T_{2}\# B_{2}$.
\end{definition}

\begin{proposition}
\label{BraidEquivalenceProp}
String links $T_1$ and $T_2$ are braid-equivalent if and only if there is an isotopy of $M$ which fixes $(\partial D^2) \x I$ and which takes $T_1$ to $T_2$.
\end{proposition}

\begin{proof}
``$\Rightarrow$": Suppose $T_1=B_{1}\# T_{2}\# B_{2}$.  Any braid $B$ in $M$ can be viewed as an isotopy $h_s, s \in I$, of points in $D^2$.  Let $D^2(1-\varepsilon)$ denote the disk of radius $(1-\varepsilon)$ centered at $(0,0)$.  By the isotopy extension theorem, $h_s$ extends to an isotopy of all of $D^2$ with support in $D^2(1-\varepsilon)$ and with $h_0=id$.  For $(x,t) \in D^2 \x I = M$, define $H_s(x,t) := (h_{st}(x), t)$.  Thus $H$ is an isotopy from the standard trivial string link to $B$.  Let $H^1, H^2$ be such isotopies corresponding to the braids $B_1, B_2$.  Take $T_2$ and stack a trivial string link on either side.  Applying $H^1, H^2$ to the left-hand and right-hand pieces of $M$ containing these trivial string links yields an isotopy as in the Proposition statement taking $T_2$ to $T_1$.

``$\Leftarrow$":  Let $H=H_s(x,t)$ be an isotopy of $M$ as in the Proposition statement.
Let $h^0, h^1$ be isotopies of $D^2$ given by $h^0_s(x)=\pi \circ H_s(x,0)$ and $h^1_s(x)=\pi \circ H_s(x,1)$, where $\pi : D^2 \x I \to D^2$ is the projection.  Note that $h^0$ and $h^1$ start at the identity map.
Take $T_1$ and stack a trivial string link $\iota$ on either side, so that $T_1$ lies in $D^2 \x [\frac{1}{4}, \frac{3}{4}]$.  Denote the result $\iota \# T_1 \# \iota$.  We will extend $H$ from the piece of $M$ containing $T_1$ to an isotopy $\widetilde{H}$ of all of $M$ which fixes $\partial M$.  In fact, let $\sigma : I \to I$ be given by $\sigma(t) = t/2 + 1/4$, and define
\[
\widetilde{H}_s(x,t) :=
\left\{
\begin{array}{ll}
(h^0_{s(4t)}(x), t) & \mbox{if $0 \leq t \leq 1/4$} \\
(id \x \sigma) \circ H_s(x, 2t - 1/2) & \mbox{if $1/4 \leq t \leq 3/4$} \\
(h^1_{s(4-4t)}(x), t) & \mbox{if $3/4 \leq t \leq 1$}
\end{array}
\right.
\]
Thus $T_2$ is equivalent as a string link to the image of $\iota \# T_1 \# \iota$ under $\widetilde{H}_1$.  Clearly the images of the copies of $\iota$ are braids, while $T_1$ is taken to $T_2$.
\end{proof}

\begin{proposition}\label{commute}
Suppose $K$ and $L$ are prime $n$-string links. If $T=K \# L=L \# K$, then at least one of the following holds:
\begin{itemize}
\item[(1)]
 $K$ and $L$ differ by a braid $B$ which commutes with both $K$ and $L$ (which implies that $K$ is braid equivalent to $L$);
\item[(2)]
Up to relabeling $K$ and $L$, $L$ is a string link with cabling annulus $A$, and $K$ has a splitting disk $F$ such that $\partial F$ is isotopic to a component of $\partial A$.
\end{itemize}
\end{proposition}

\begin{proof}
Suppose $T$ is a string link in $M=D^2 \x I$ such that $T$ is equivalent to both $K \# L$ and $L \# K$ for prime $n$-string links $K$ and $L$.
Let $D$ be the decomposing disk that decomposes $T$ as $K\# L$ and let $E$ be the decomposing disk that decomposes $T$ as $L\# K$. By Lemma \ref{incompdecomp} and the definition of decomposing disk, both $D$ and $E$ are incompressible and non-boundary-parallel. Since $\partial D$ is isotopic to $\partial E$, we can isotope $E$ so that $\partial D \cap \partial E = \emptyset$ while simultaneously demanding that $D=D^2 \times \{\frac{1}{2}\}\subset M$ (see Figure \ref{TwoDecomposingDisks} for a particular example). This last condition is not strictly necessary for the proof, but it simplifies the pictures in Figures \ref{CaseA} and \ref{CaseB}.

\begin{figure}[h!]
\begin{picture}(250,117)
\put(0,0){\includegraphics[scale=.9]{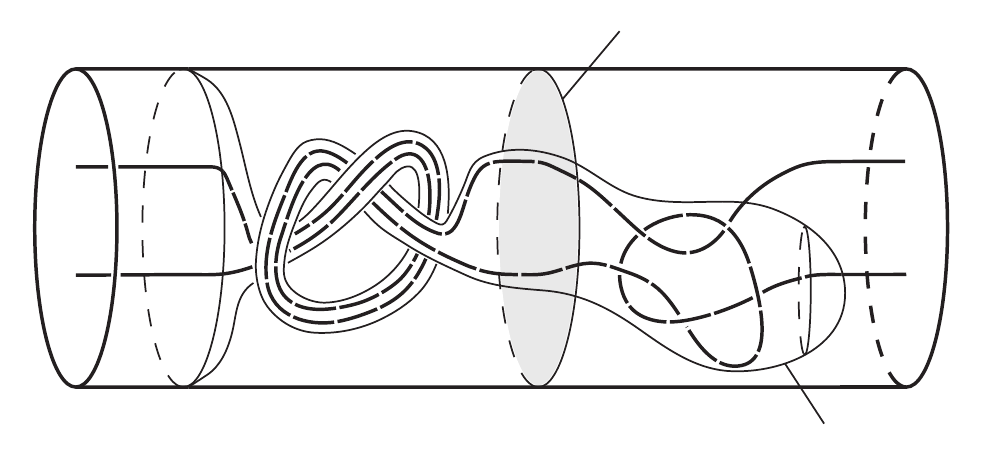}}
\put(161,108){$D$}
\put(213,0){$E$}
\end{picture}
\caption{Decomposing Disks for $K \# L$ and $L \# K$}
\label{TwoDecomposingDisks}
\end{figure}

In general, the intersection of (compact) surfaces $D$ and $E$ with $\partial D \cap \partial E = \emptyset$ is a finite collection of circles.  We will let $|D \cap E|$ denote the number of these circles of intersection.
Subject to the constraints on $D$ and $E$ in the previous paragraph, isotope $E$ to minimize $|D\cap E|$.

If $|D\cap E|=0$, then, since both $L$ and $K$ are prime, $D$ is isotopic to $E$.  Thus the part of $T$ in the region of $M$ bounded between $D$ and $E$ is a braid $B$.
This implies furthermore that $K \# B = L$ and $B \# K =L$.  (Alternatively, replacing $B$ by $B^{-1}$, $K = L \# B$ and $K = B \# L$ for some braid $B$.)  Thus we arrive at condition (1) in the Proposition statement.  In the remaining cases, we will arrive at either condition (2) or a contradiction.

Suppose $|D\cap E|>0$.
Let $\{\alpha_i\}$ be the collection of all innermost curves of $D\cap E$ in $E$ and let $D_i, E_i$ be the possibly punctured disks that $\alpha_i$ bounds in $D, E$ respectively.
Since $T$ is a string links, then $|D_i\cap T|=|E_i \cap T|$.   So if one of $\{D_i, E_i\}$ has zero punctures, then so does the other.  In that case, $D_i\cup E_i$ bounds a 3-ball, and $D_i$ can be pushed across $E_i$ to eliminate $\alpha_i$, contradicting the minimality of $|D \cap E|$.  Thus, every (innermost) component of $D\cap E$ in $E$ encloses at least one puncture.

Let $D$ separate $M$ into $M_{K}^{D}$ and $M_{L}^{D}$. Since $\alpha_i$ is innermost in $E$, then $D$ is disjoint from the interior of $E_{i}$ for each $i$ and $E_{i}$ is properly embedded in one of $M_{K}^{D}$ or $M_{L}^{D}$ for each $i$.

\noindent \textbf{Case A:}  Assume $|\{\alpha_i\}|\geq 2$. Let $E_j$ and $E_k$ be distinct innermost disks in $E$. Let $D_j$ and $D_k$ be the punctured disks in $D$ that $\alpha_j$ and $\alpha_k$ bound respectively. Since $|D\cap E|$ has been minimized, $D_j$ is not isotopic to $E_j$ and $D_k$ is not isotopic to $E_k$.
Note that since $E_j,E_k,D_j$ and $D_k$ are all subdisks of $E$ and $D$, and since both $D$ and $E$ meet every strand of $T$ exactly once, then each of $E_j,E_k,D_j$ and $D_k$ must meet each strand of $T$ at most once.

See Figure \ref{CaseA} for illustrations of the following subcases.

\noindent \textbf{Case A1:}  Suppose both $E_j$ and $E_k$ are properly embedded in $M_{L}^{D}$.
\noindent \textbf{\emph{Sub-case A1a}}: Suppose $D_j$ and $D_k$ do not intersect in $D$ then the $n$-punctured disk $H$ that is the union of $E_j$ and the annulus bounded by $\alpha_j$ and $\partial D$ in $D$ is not isotopic to $D$ since $E_j$ is not isotopic to $D_j$, and $H$ is not isotopic to $\partial_+ M_{L}^{D}$ since $E_k$ is not isotopic to $D_k$. This is a contradiction to $L$ being prime.

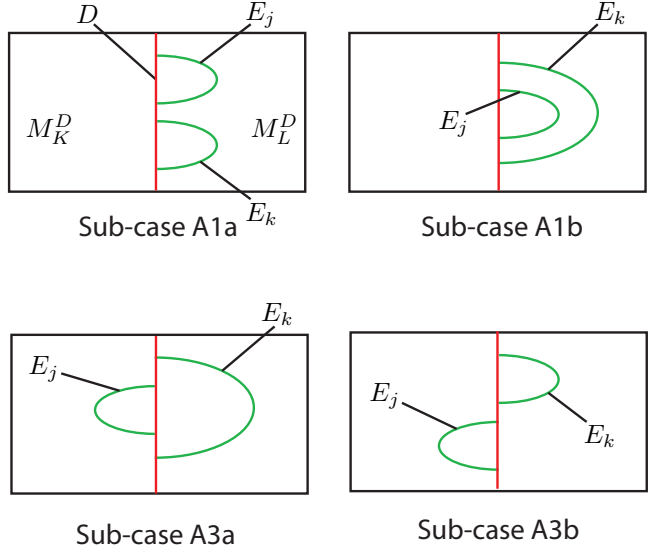
\begin{figure}[h!]
\begin{picture}(250,200)
\put(0,0){\includegraphics[scale=.9]{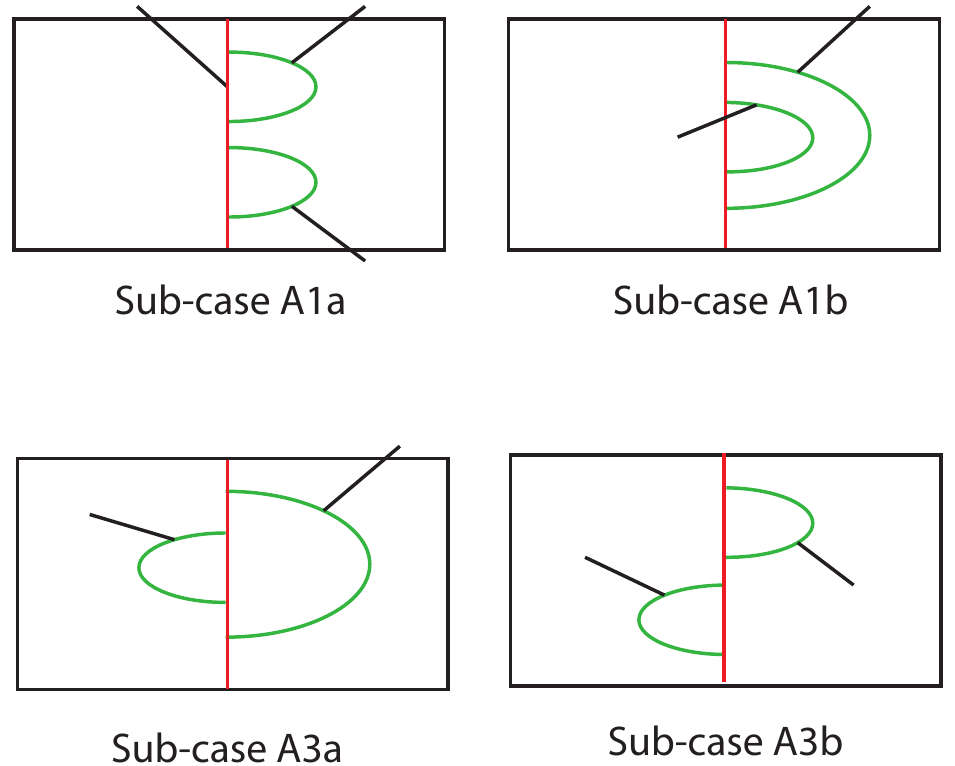}}
\put(10,155){\footnotesize$M^D_K$}
\put(95,155){\footnotesize$M^D_L$}
\put(29,198){\footnotesize$D$}
\put(93,199){\footnotesize$E_j$}
\put(93,123){\footnotesize$E_k$}
\put(225,199){\footnotesize$E_k$}
\put(166,158){\footnotesize$E_j$}
\put(11,65){\footnotesize$E_j$}
\put(99,85){\footnotesize$E_k$}
\put(221,39){\footnotesize$E_k$}
\put(140,55){\footnotesize$E_j$}
\end{picture}
\caption{Diagrammatic representations of  the subcases in Case A.  A rectangle is used to represent $D^2 \times I$; $D$ is represented by a red line, with $M^D_K$ on the left and $M^D_L$ on the right; and the subdisks $E_j$ and $E_k$ of $E$ are represented by green lines.}
\label{CaseA}
\end{figure}

\noindent \textbf{\emph{Sub-case A1b:}} Suppose $D_j$ and $D_k$ do intersect in $D$. Up to relabeling, we can assume that $D_j\subset D_k$. Notice that each of $E_j$ and $D_j$  meet each strand of $T$ at most once.  Thus every arc of $T$ contained in the 3-ball whose boundary is $D_j \cup E_j$ has one endpoint in $D_j$ and one endpoint in $E_j$.  By Alexander's Theorem, the punctured disks $E_j \cup E_k$ separate $M_{L}^{D}$ into three 3-ball components.  Examine an arc $\beta \subset L$ in the component which is incident to both $E_j$ and $E_k$ such that $\beta$ has an end point on $E_j$. There are three possibilities for the other endpoint of $\beta$: it is contained in $E_j$, $E_k$, or $D$.
If this other endpoint of $\beta$ is contained in either $E_j$ or $E_k$, then we immediately contradict that $\beta$ intersects $E$ only once.
If this other endpoint of $\beta$ is contained in $D$, then the strand of $L$ that contains $\beta$ has both endpoints on $D$, a contradiction to $L$ being a string link.
Hence, in each case we derive a contradiction.

\noindent \textbf{Case A2:} Suppose both $E_j$ and $E_k$ are properly embedded in $M_{K}^{D}$. The proof in this case is similar to the proof in Case A1.

\noindent \textbf{Case A3:} Suppose $E_j$ is properly embedded in $M_{K}^{D}$ and $E_k$ is properly embedded in $M_{L}^{D}$.

\noindent \textbf{\emph{Sub-case A3a:}} Suppose $D_j$ and $D_k$ intersect in $D$, then we can assume $D_j\subset D_k$. Let $p$ be a point of intersection between $T$ and $D_j$.  (We know such a point exists since no curve of $D\cap E$ bounds an unpunctured disk in $D$.)  The point $p$ is the endpoint of an arc $\alpha$ in $K$ and an arc $\beta$ in $L$. Since the other endpoint of $\alpha$ is contained in $\partial_- M_{K}^{D}$, then $E_j$ separates the endpoints of $\alpha$ and $\alpha \cap E_j \neq \emptyset$. Since the other endpoint of $\beta$ is contained in $\partial_+ M_{L}^{D}$, then $E_k$ separates the endpoints of $\beta$ and $\beta \cap E_k \neq \emptyset$. Hence the arc $\alpha \cup \beta$ in $T$ meets $E$ twice, a contradiction.

\noindent \textbf{\emph{Sub-case A3b:}} Suppose $D_j$ and $D_k$ do not intersect in $D$, then examine the $n$-punctured disk $D'$ that is the union of $E_j$ and the annulus in $D$ with boundary $\alpha_j \cup \partial D$.  Since $K$ is prime and $E_j$ is not isotopic to $D_j$, then $D'$ must be isotopic to $\partial_- M_{K}^{D}$ via an isotopy $\iota$.  Moreover, $\iota$ can be taken to be an embedding.  This identifies the part of $M$ between $D'$ and $\partial_- M^D_K$ with the product $D^2 \x I$ such that $K$ meets this product in $I$-fibers.
Now for any curve $\gamma$ in $D$ which is disjoint from the interior of $D_j$ and which encloses a puncture in $D$, the image of the restriction of $\iota$ to $\gamma$ is a cabling annulus in $M_{K}^{D}$ with $\gamma$ a boundary component.  Thus the restriction of $\iota$ to $\alpha_k$ has image a cabling annulus in $M_{K}^{D}$ with $\alpha_k$ as a boundary component. Hence, $E_k$ is a splitting disk for $L$ such that $\alpha_k$ bounds a cabling annulus in $M_{K}^{D}$, which is what we wanted to show.

\noindent \textbf{Case A4:} Suppose $E_j$ is properly embedded in $M_{L}^{D}$ and $E_k$ is properly embedded in $M_{K}^{D}$. This case follows from the proof of Case A3.

\noindent \textbf{Case B:} It remains to consider the case that there is a unique innermost curve $\alpha_1$ of $D\cap E$ in $E$ bounding a punctured disk $E_1$ in $E$. Let $D_1$ be the punctured disk $\alpha_1$ bounds in $D$. Since there is a unique innermost curve $\alpha_1$ of $D\cap E$ in $E$, every other curve in $D\cap E$ encloses $\alpha_1$ in $E$, and hence there is also a unique curve $\beta$ which is outermost in $E$. Let $A_E$ be the possibly punctured annulus in $E$ with boundary $\beta \cup \partial E$. Let $A_D$ be the possibly punctured annulus in $D$ with boundary $\beta \cup \partial D$. $E_1$ is not isotopic to $D_1$, and $A_E$ is not isotopic to $A_D$, as otherwise we could decrease $|D\cap E|$. As noted previously, $|E_1 \cap T|=|D_1 \cap T|$ and, similarly, $|A_E \cap T|=|A_D \cap T|$. Without loss of generality, assume $A_E\subset M_{K}^{D}$. There are several cases to consider.

See Figure \ref{CaseB} for illustrations of the following subcases.

\noindent \textbf{Case B1:} $E_1$ is contained in  $M_{K}^{D}$, and $\alpha_1$ is disjoint from $A_D$. The $n$-punctured disk $H$ in $M_{K}^{D}$ that is the union of $E_1$ and the annulus bounded by $\partial E_1$ and $\partial D$ in $D$ is not isotopic to $D$ since $E_1$ is not isotopic to $D_1$ and $H$ is not isotopic to $\partial_- M_{K}^{D}$ since $A_D$ is not isotopic to $A_E$. This is a contradiction to $K$ being prime.

\begin{figure}[h!]
\begin{picture}(250,210)
\put(0,0){\includegraphics[scale=.9]{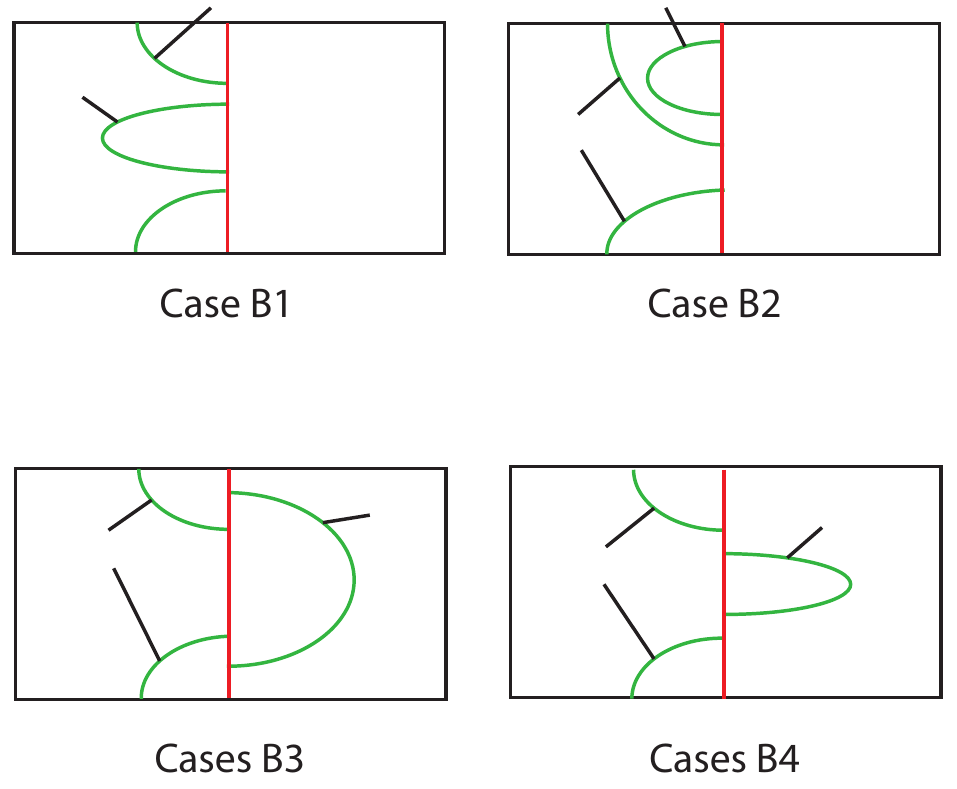}}
\put(51,205){\footnotesize$A_E$}
\put(10,179){\footnotesize$E_1$}
\put(167,205){\footnotesize$E_1$}
\put(143,169){\footnotesize$A_E$}
\put(20,60){\footnotesize$A_E$}
\put(97,68){\footnotesize$E_1$}
\put(150,57){\footnotesize$A_E$}
\put(213,69){\footnotesize$E_1$}
\end{picture}
\caption{Diagrammatic representations of  the subcases in Case B.  A rectangle is used to represent $D^2 \times I$; $D$ is represented by a red line, with $M^D_K$ on the left and $M_D^L$ on the right; and $D_1$ and $A_E$ are represented by green lines.}
\label{CaseB}
\end{figure}
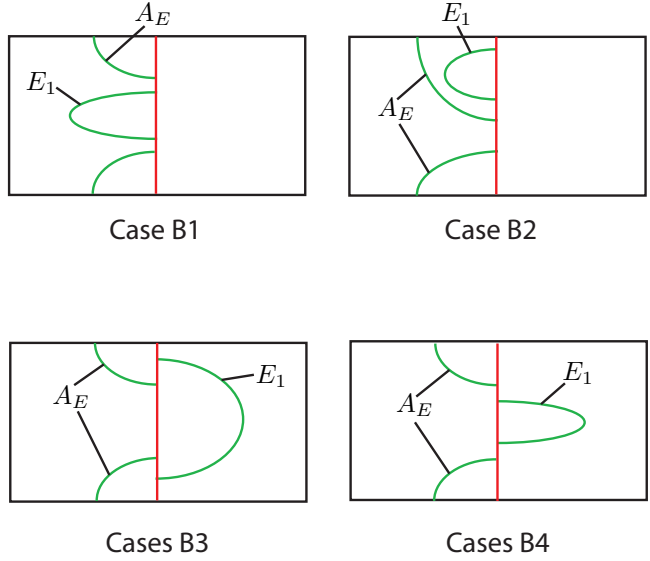

\noindent \textbf{Case B2:} $E_1$ is contained in  $M_{K}^{D}$, and $\alpha_1$ is contained in $A_D$.  The annulus $A_E$ separates $M^D_K$ into two pieces. (This can be seen by doubling the 3-ball $M^D_K$ and the annulus along their boundaries and applying the Generalized Jordan Curve Theorem \cite{GP}.)  Since $E$ is embedded, $E_1$ is contained in one of these pieces, necessarily the one incident to $A_D$.  This implies that $D_1 \subset A_D$.  Then any arc of $K$ that intersects $E_1$ is forced to also intersect $A_E$, a contradiction to the fact that every arc of $T$ meets $E$ exactly once.

\noindent \textbf{Case B3:} $E_1$ is contained in  $M_{L}^{D}$, and $\alpha_1$ cannot be isotoped to be disjoint from $A_D$. Here $A_D$ and $D_1$ are forced to meet $T$ in a common puncture $p$. The arc of $K$ with endpoint $p$ must meet $A_E$ since $A_E$ separates the endpoints of this arc. The arc of $L$ with endpoint $p$ must meet $E_1$ since $E_1$ separates the endpoints of this arc. This is a contradiction to the fact that $E$ meets every arc of $T$ exactly once.

\noindent \textbf{Case B4:} $E_1$ is contained in  $M_{L}^{D}$, and $\alpha_1$ can be isotoped to be disjoint from $A_D$.
Note that $E_1$ is a splitting disk for $L$.  Since $A_E$ is not isotopic to $A_D$ and $K$ is prime, the $n$-punctured disk $A_E \cup (D-A_D)$ is isotopic to $\partial_- M_{K}^{D}$. As in the proof of Case A3b, this isotopy gives rise to a product structure on the complementary component of $A_E$ in $M_{K}^{D}$ that contains $\alpha_1$. This product structure implies $\alpha_1$ is a boundary component of a cabling annulus for $K$, which is what we wanted to show.

\end{proof}

\begin{remark}
By condition (1) of Proposition \ref{commute}, understanding when braids commute with string links would provide a more thorough understanding of the necessary conditions for commutativity of string links.  We plan to pursue this in future work.
\end{remark}

The following essentially provides the converse to the above Proposition.

\begin{proposition}\label{commutecon}
Suppose $T= K \# L$, where $K$ and $L$ are prime, $n$-string links.  Let $D$ be the decomposing disk that decomposes $T$ as $K\# L$.  If $L$ is a string link with cabling annulus $A$ and $K$ has a splitting disk $F$ such that, after forming $K \# L$, $\partial F$ is isotopic to a component of $\partial A$ in $\partial_+ M^D_K$, then $K \# L$ is braid equivalent to $L \# K$.
\end{proposition}

It is clear that the roles of $K,L$ can be reversed, i.e., that if $L$ has a splitting disk $F$, $K$ has a cabling annulus $A$, and $\partial F$ is isotopic to a component of $\partial A$ in $\partial_- M^D_L$, then the same conclusion holds.

\begin{proof}
Here we describe an isotopy taking $K \# L$ to $L \# K$ up to braid equivalence. An example of the splitting disk and cabling annulus that define the isotopy are depicted in Figure \ref{TwoDecomposingDisks} as the portions of $E$ on either side of $D$. Additionally, a version of this isotopy is depicted in Figure \ref{CableCommute}.

The portion of $L \# K$ in the 3-ball bounded by $F$ and a subdisk of $\partial_+ M^D_{K}$ in $M^D_{K}$ can be isotoped through the complementary component of $A$ in $M^D_L$ homeomorphic to $D^2\times I$ so that it lies inside an $\varepsilon$ neighborhood of $\partial_+ M^D_L$.  By Lemma \ref{splitdiskprime}, after this isotopy, the portion of $K \# L$ in $M^D_K$ is a braid. Additionally, since, before the isotopy, $T$ met the complementary component of $A$ in $M^D_L$ homeomorphic to $D^2\times I$ in a braid, then, after the isotopy, the portion of $K \# L$ in an $\varepsilon$ neighborhood of $\partial_+ M^D_L$ is braid equivalent to $K$. Similarly, after the isotopy, the portion of $K\# L$ in $M^D_L$ that meets the $D^2\times I$ complementary component of $A$ in $M^D_L$ outside of an $\varepsilon$ neighborhood of $\partial_+ M^D_L$ remains a braid. Simultaneously, the portion of $K\# L$ in $M^D_L$, but outside of this $D^2\times I$ complementary component of $A$ is fixed by this isotopy. Hence, $K \# L$ is braid-equivalent to $L \# K$.
\end{proof}

\begin{corollary}
Let $T$ be an $n$-string link such that $T= K\# L=P\# Q$ $K$, $L$, $P$, $Q$ be prime, $n$-string links.  Then $K$ is braid-equivalent to $P$ and $L$ is braid-equivalent to $Q$ or $K$ is braid-equivalent to $Q$ and $L$ is braid equivalent to $P$.
\end{corollary}

\begin{proof}
This follows an argument very similar to the proof of Proposition \ref{commute} and the proof of Proposition \ref{commutecon}.  Let $D$ be disk that decomposes $T$ as $K\# L$, and let $E$ be the disk that decomposes $T$ as $P\# Q$. If $D$ can be isotoped to be disjoint from $E$, then by Proposition \ref{BraidEquivalenceProp}, $K$ is braid-equivalent to $P$ and $L$ is braid equivalent to $Q$. If $D$ can not be isotoped to be disjoint from $E$, then we arrive at a contradiction exactly as in the proof of Proposition \ref{commute}, except for case B4 and sub-case A3b, which require only slight modification, as follows.

Recall the terminology from the proof of Proposition \ref{commute}. In case B4, $E_1$ is contained in  $M_{L}^{D}$, and $\alpha_1$ can be isotoped to be disjoint from $A_D$. By the argument presented in Proposition \ref{commute}, we can conclude that $\alpha_1$ is simultaneously a boundary component of a cabling annulus for $K$ and the boundary of a splitting disk for $L$.  After applying the proof of Proposition \ref{commutecon}, we see that this implies that $K$ is braid-equivalent to $Q$ and $L$ is braid equivalent to $P$.

In the second part of case A3, $E_j$ is properly embedded in $M_{K}^{D}$, $E_k$ is properly embedded in $M_{L}^{D}$, and $D_j$ and $D_k$ do not intersect in $D$. We can conclude that $E_k$ is a splitting disk for $L$ such that $\partial E_k$ bounds a cabling annulus in $M_{K}^{D}$. After applying the proof of Proposition \ref{commutecon} with the roles of $K$ and $L$ switched, we see that this implies that $K$ is braid-equivalent to $Q$ and $L$ is braid equivalent to $P$.
\end{proof}

The previous corollary immediately implies the following one.  Although for arbitrary $n$, the monoid structure does not respect braid-equivalence, this statement says roughly that ``string links modulo braid equivalence are cancellative."

\begin{corollary}
\label{cancellative}
Let $K$, $L$, $P$, $Q$ be prime, $n$-string links. $K \# L = K \# Q$ implies $L$ is braid-equivalent to $Q$, and $K \# L = P \# L$ implies $K$ is braid-equivalent to $P$.
\qed
\end{corollary}


\begin{question}
Can the conclusions in Corollary \ref{cancellative} be improved to ``$L$ is isotopic to $Q$" and ``$K$ is isotopic to $P$"?
\end{question}

By the final result of this paper, Corollary \ref{cor:main}, the answer is affirmative for $n=2$.

We can improve slightly on Corollary \ref{cancellative} in the case that $Q$ (or $P$) is itself a braid.  Recall that a \emph{link-homotopy} of a string link $K$ is a homotopy where distinct strands of $K$ never intersect.  In other words, a link-homotopy is like an isotopy, except that each strand may intersect itself.

\begin{proposition}
Let $K$ be any string link.  Suppose $K \# B = K \# B'$ for braids $B, B'$.  Then $B$ is link-homotopic to $B'$.  (Similarly, $B \# K = B' \# K$ implies that $B,B'$ are link-homotopic.)
\end{proposition}
\begin{proof}
In \cite[pp. 394-397]{HL}, Habegger and Lin showed that there is a homomorphism $\varphi$ from the monoid of $n$-string links to the group of automorphisms of $RF(n)$, the ``reduced" free group on $n$ letters.  This map factors through the monoid of string links up to link-homotopy, on which it is an isomorphism \cite[Theorem 1.7]{HL}.  From $\varphi(K) \varphi(B) = \varphi(K) \varphi(B')$ and the fact this equality is in a group, we conclude that $\varphi(B)=\varphi(B')$.  Since $\varphi$ is an isomorphism on the monoid of string links up to link-homotopy, the desired result follows.
\end{proof}

In particular, the only braids which fix a string link are those which are link-homotopic to the trivial braid.

\subsection{Consequences for the 2-string link monoid}

\begin{definition}
An $n$-string link $T$ with a non-boundary-parallel cabling annulus such that all the strands of $T$ are contained in a complementary component of the annulus is called a \emph{one-strand cable}.  By the definition of a cabling annulus (Definition \ref{CablingAnnulusDefn}), this component is necessarily the one homeomorphic to $D^2 \x I$.
\end{definition}

\begin{corollary}
\label{cor:ifcentral}
If $K\subset M\cong D^2\times I$ is a prime element in the center of the monoid of $n$-string links, then either $K$ is a one-strand cable or $M$ contains an essential 2-punctured sphere.
\end{corollary}

\begin{proof}
Sketch: It is easy to show that for any $n>1$, there exists an $n$-string link $K$ such that the only cabling annulus for $K$ is the boundary of a regular neighborhood of a single strand of $K$ and $K$ admits no nontrivial splitting disks. An example of such a string link when $n=2$ is given in the upper left hand side of Figure \ref{StringlinkStacking}. Hence, by the Proposition \ref{commute}, the only $n$-string links that can commute with $K$ are one-strand cables or string links with a 1-punctured splitting disk. However, a $n$-string link $L$ in $M$ with a 1-punctured splitting disk $D$ naturally contains an essential 2-punctured sphere which is the union of $D$ and the 1-punctured disk that $\partial D$ bounds in $\partial M$.
\end{proof}

\begin{corollary}\label{cor:central}
The prime elements of the center of the monoid of pure $2$-string links consist of

\begin{enumerate}
\item prime split links
\item prime one-strand cables
\end{enumerate}
\end{corollary}

\begin{proof}
By Corollary \ref{cor:ifcentral}, the prime central elements are contained in the set of prime one-strand cables and prime split string links. Conversely, Proposition \ref{commutecon} applied to the case of 2-string links implies that one-strand cables and split string links are central, as illustrated in Figure \ref{SplitCommute} and Figure \ref{CableCommute}.
\end{proof}

\begin{figure}[h!]
\begin{picture}(315,160)
\put(5,10){\includegraphics[scale=.7]{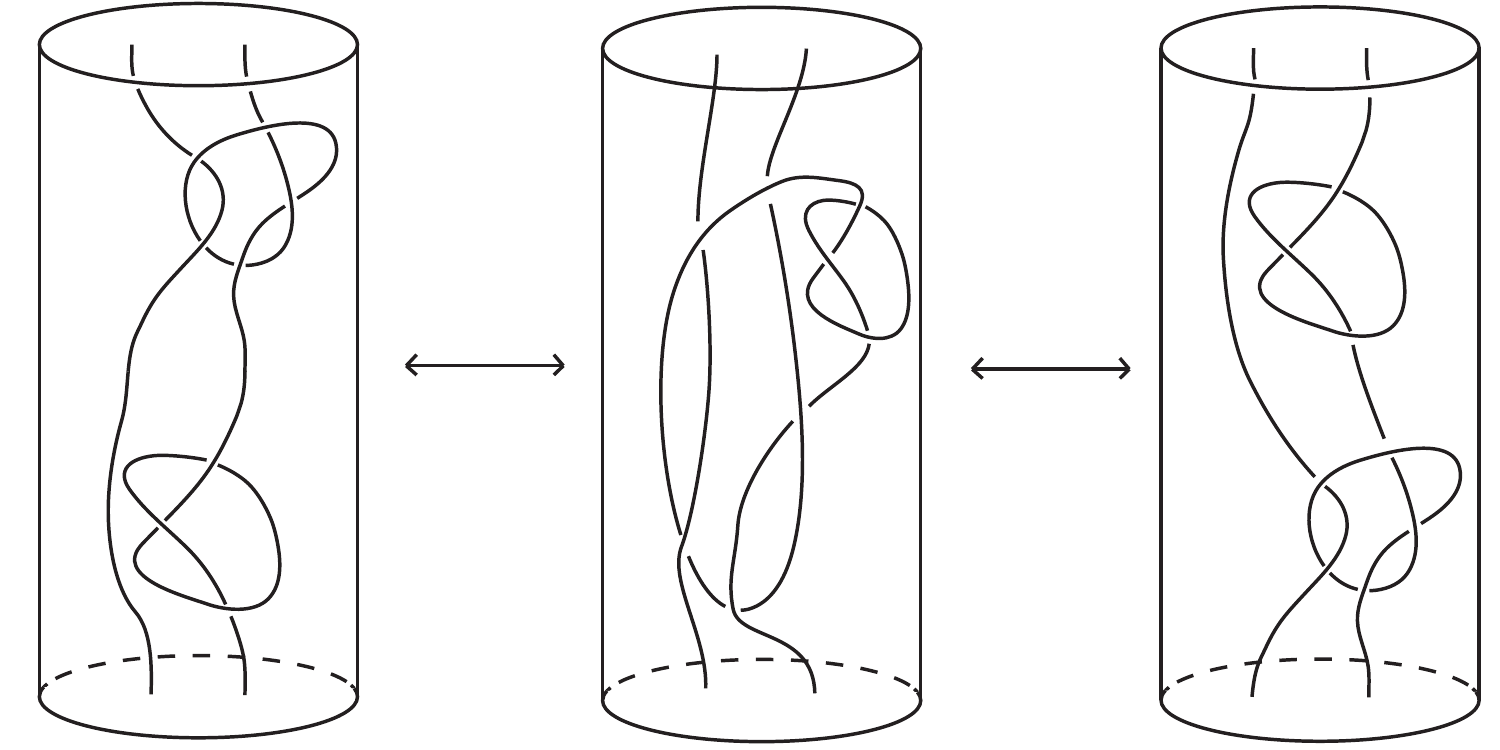}}
\put(29,0){$T_1 \# T_2$}
\put(259,0){$T_2 \# T_1$}
\end{picture}
\caption{Commuting Split links}
\label{SplitCommute}
\end{figure}

\begin{figure}[h!]
\begin{picture}(315,160)
\put(5,10){\includegraphics[scale=.7]{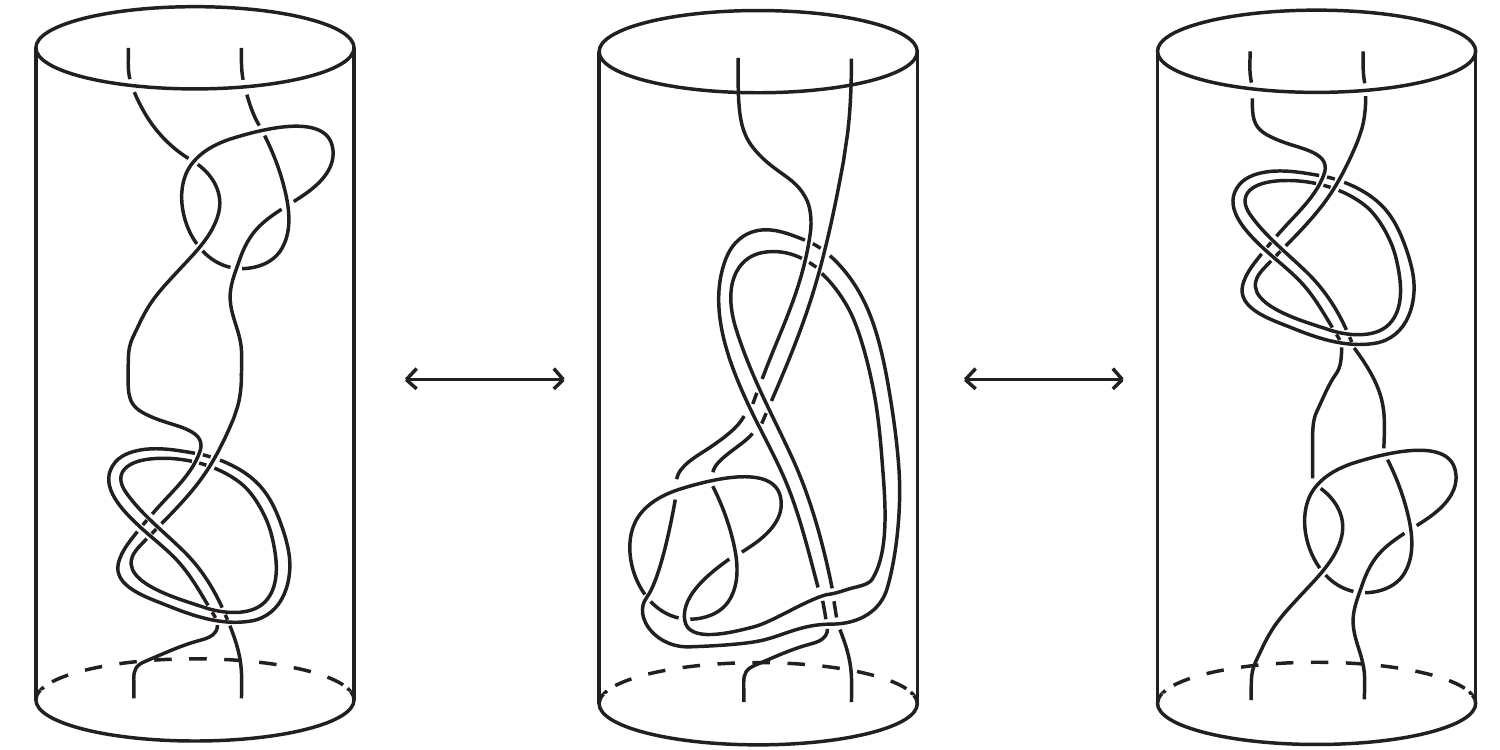}}
\put(29,0){$T_1 \# T_2$}
\put(259,0){$T_2 \# T_1$}
\end{picture}
\caption{Commuting one-strand cable}
\label{CableCommute}
\end{figure}

We now show that the monoid of all 2-string links splits as braids and the remaining quotient by braids.  Let $\LL_2$ denote the monoid of (isotopy classes of) 2-string links.

\begin{proposition}\label{prop:LL}
There is a splitting $\LL_2 \cong \LL^0_2 \oplus \Z$, where the $\Z$ factor corresponds to the braids in $\LL_2$.
\end{proposition}

\begin{proof}
On the submonoid $\LL^{pure}_{2} \subset \LL_2$ of pure string links, we have a map $2\ell: \LL^{pure}_{2} \to \Z$, given by twice the linking number.  We can extend this to all string links as follows:
Let $\tau$ be the string link obtained by composing a standard unlink $I \sqcup I \hookrightarrow D^2 \times I$ with the map $D^2 \times I \to D^2 \times I$ given by $(z,t) \mapsto (e^{\pi i t} z, t)$.  (Thus $\tau^2$ has  linking number $+1$.)
For a non-pure 2-string link $L$, define $2\ell$ by twice the linking number of $L\# \tau$ minus one.  This gives a monoid homomorphism $2\ell: \LL_2 \to \Z$.  On the other hand, we also have a monoid homomorphism $b: \Z \to \LL_2$ given by $b(n)= \tau^n$, whose image is clearly the braids in $\LL_2$.  The composite $2\ell \circ b : \Z \to \Z$ is the identity, so $b$ gives a right-splitting of the short exact sequence
\[
0 \to \mathrm{ker} (2\ell) \to \LL_2 \to \Z \to 0.
\]
Thus $\LL_2 \cong \mathrm{ker} (2\ell) \oplus \Z$, where the $\Z$ factor is the braids in $\LL_2$.  So if we let $\LL^0_2$ denote the quotient of $\LL_2$ by the braids, then $\LL_2 \cong \LL^0_2 \oplus \Z$.  (We chose this notation because $\LL^0_2$ is also isomorphic to $\mathrm{ker}(2\ell)$, which is the 2-string links with linking number zero.)
\end{proof}

It immediately follows that pure braids are central in $\LL_2$, a fact which we invite the reader to visualize directly.

\section{Prime decomposition for 2-string links}

We now turn our attention to pure 2-string links, in which case we can prove the following prime decomposition theorem.

\begin{theorem}
\label{2LinkPrimeDecomp}
A pure 2-string link $L$ has a prime decomposition $L=L_1 \# ... \# L_n$, where each $L_i$ is a prime pure string link.  Such a decomposition is unique up to reordering the factors in the center and up to multiplication by units (pure braids).
\end{theorem}

\begin{proof}
The following theorem of Freedman and Freedman implies the \emph{existence} statement.
In fact, we get the existence statement for $n$-string links for any $n$, i.e., any $n$-string link has a finite decomposition into prime factors.  We just apply the theorem below with $M$ the exterior (i.e., complement of an open neighborhood) of the $n$-string link, $F_1, . . . , F_k$ is any collection of decomposing disks restricted to the exterior, and $b=n$.

\begin{theorem}\label{Freedman}~\cite{FF} Let $M$ be a compact 3-manifold with boundary and $b$ an integer greater than
zero. There is a constant $c(M, b)$ so that if $F_1, . . . , F_k$, $k > c$, is a collection of incompressible
surfaces such that all the Betti numbers $b_{1}F_{i} < b$, $1 \leq i \leq k$, and no $F_i$, $1 \leq i \leq k$, is a boundary
parallel annulus or a boundary parallel disk, then at least two members $F_i$ and $F_j$ are parallel.
\end{theorem}

Before proving the  \emph{uniqueness} statement, we prove a lemma.

\begin{lemma}[Splitting Disk Lemma]
\label{SplitDiskLemma}
Suppose an $n$-string link $L$ in $M$ with strands labeled $1$ to $n$ and $L$ has a nontrivial $k$-punctured splitting disk with punctures corresponding to strands labeled $1,...,k$.  Then any prime decomposition of $L$ must have a prime factor $K$ which is a split string link.
Furthermore, $K$ has a splitting disk that meets each of the strands labeled $1,...,k$ at most once and is disjoint from all other strands.
\end{lemma}
\begin{proof}
Suppose $L$ has a nontrivial splitting disk $D$ with punctures corresponding to the first $k$ strands and whose boundary (without loss of generality) is contained in $\bdd_-M$.  By possibly replacing $D$ by a ``smaller" disk (lying between $D$ and $\bdd_- M$), we may assume that the $k$-string link $L'$ contained in the 3-ball with boundary $D$ union a k-punctured subdisk of $\partial_- M$ is prime. Suppose we have a prime decomposition of $L$, and let $E_1,...,E_n$ be the decomposing disks for this decomposition.  Isotope the $E_i$ so as to minimize the number of components of the intersection with $D$. Let $\alpha$ be a circle of intersection between the union of the $E_i$ and $D$ which is innermost in $D$. The loop $\alpha$ bounds a punctured disk $D_{\alpha}$ in $D$ and a punctured disk $E_{\alpha}$ in some $E_j$ such that the interior of $D_{\alpha}$ is disjoint from the union of the $E_i$. As the intersection between the $E_i$ and $D$ is was taken to be minimal, $D_{\alpha}$ is not isotopic to $E_{\alpha}$. Since $D$ is incident to only strands labeled $1,...k$, so is $D_{\alpha}$. Let $M_j$ be the complementary component of the union of the $E_i$ in $M$ that contains $D_{\alpha}$ and let $K_j$ be the portion of $L$ in $M_j$. Then $D_{\alpha}$ is a splitting disk for $K_j$ in $M_j$ that meets each of the strands labeled $1,...,k$ at most once and is disjoint from all other strands.

\end{proof}

Returning to the proof of the \emph{uniqueness} statement of the Theorem, suppose we have a prime decomposition of a 2-string link $L = L_1 \# ... \# L_n$.  Let $T_1$ be the product of all the $L_i$ which are split links with the first component unknotted.  Similarly, let $T_2$ be the product of all the $L_i$ which are split links with the second component unknotted.  Let $T_3$ be the product of all the $L_i$ which are one-strand cables.  Let $T_4$ be the product of the remaining $L_i$ in the given prime decomposition of $L$.  Since all the $L_i$ in $T_1,T_2,T_3$ are central by Corollary \ref{cor:central}, we can write $L=T_1\# T_2 \# T_3 \# T_4$.

Given another prime decomposition $L=L'_1\# ... \# L'_m$, we similarly can write $L=T'_1 \# ... \# T'_4$, where each $T'_i$ is a product of $L'_i$'s of the same type as those in $T_i$.  The first step is to show that $T_i = T'_i$ modulo braid equivalence for each $i$.  The prime decomposition theorem for knots will then imply equality of the factors in $T_i$ and $T'_i$ modulo braid equivalence for $i=1,2,3$.  It will then remain to prove that the factors in $T_4$ and $T_4'$ agree up to braid equivalence. Since 2-string pure braids are the units, the theorem will follow.

\textbf{Step 1:}
We will show that $T_1$ and $T'_1$ agree up to stacking with a pure braid.  Let $D_1$ be the 2-punctured decomposing disk properly embedded in $M$ and separating $T_1$ from $T_2\# ...\# T_4$.  Similarly, let $E_1$ be such a 2-punctured decomposing disk separating $T'_1$ from $T_2' \# ... \# T_4'$ in $L$.  Since $T_1$ is split, $D_1$ is isotopic to the union of a once-punctured annulus in $\bdd_- M$ with a once-punctured, properly embedded disk $D$ whose boundary is contained in $\bdd_- M$.  Similarly, $E_1$ is isotopic to the union of a once-punctured annulus in $\bdd_- M$ with a 1-punctured disk $E$ whose boundary is contained in $\bdd_- M$.

By an isotopy, we can take $\bdd D$ and $\bdd E$ to be disjoint in $\bdd_- M$. Perform a further proper isotopy of $D$ and $E$ so as to minimize the number of components of $D \cap E$.  Suppose that this number is nonzero.  Let $\alpha$ be a component of $D \cap E$ which is innermost in $E$. Let $D_\alpha$ be the possibly punctured disk in $D$ bounded by $\alpha$, and let  $E_\alpha$ be the possibly punctured disk in $E$ bounded by $\alpha$.  Now $D$ and $E$ each have exactly one puncture.  Thus, $D_\alpha$ and $E_\alpha$ must both have either zero punctures or one puncture.  In the case of zero, $D_\alpha \cup E_\alpha$ bounds a 3-ball, and we can get rid of $\alpha$ by pushing $D_\alpha$ across $E_\alpha$, thus contradicting the minimality of $|D\cap E|$.  Hence $D_\alpha$ and $E_\alpha$ each have one puncture.

Since $D$ and $E$ themselves have just one puncture, the argument just made shows that $\alpha$ is the unique curve in $D\cap E$ which is innermost in $E$.  Reversing the roles of $D$ and $E$ shows that among the curves in $D\cap E$, there is a unique curve $\beta$ that is innermost in $D$.  Let $D_\beta$ denote the punctured  disk that $\beta$ bounds in $D$.  We can so far see that the circles in  $D\cap E$ are all concentric in both $D$ and $E$.

In other words, the components of $D\cap E$ separate each of $D$ and $E$ into $|D\cap E|+1$ regions which look like regions of a dartboard as shown in Figure \ref{DartBoard}.  We now label each region of $D$ by a ``1" and ``2" according as the region is in $M^{E_1}_{T'_1}$ or $M^{E_1}_{T'_2 \# T'_3 \# T'_4}$.  Similarly, we label each region of $E$ by a ``1" or ``2" according as the region is in $M^{D_1}_{T_1}$ or $M^{D_1}_{T_2 \# T_3 \# T_4}$.  The labels of the outermost regions of $D$ and $E$ must be different.  In each of $D$ and $E$, the labels of regions must alternate.  Since $D$ and $E$ have the same number of regions, the labels of $D_\beta$ and $E_\alpha$ must be different.  Hence one of these disks must have the label ``2".  Without loss of generality, suppose it is $E_\alpha$ that has the label ``2".  Then $E_\alpha$ is a once-punctured splitting disk in $M^{D_1}_{T_2 \# T_3 \# T_4}$.  (If it is not a splitting disk, we contradict the minimality of $|D\cap E|$.)  By the Splitting Disk Lemma, one of the prime factors of $T_2 \# T_3 \# T_4$ must be a split link with the first strand unknotted.  This contradicts the definition of the $T_i$.  We conclude that $D$ and $E$ can be isotoped to be disjoint.

\begin{figure}[h!]
\begin{picture}(175,100)
\put(37,0){\includegraphics[scale=.7]{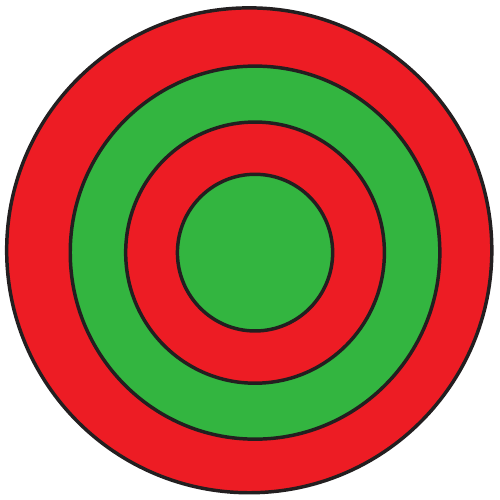}}
\end{picture}
\caption{}
\label{DartBoard}
\end{figure}

Suppose, the disjoint punctured disks $D$ and $E$ are not isotopic to each other. Since both $\partial D$ and $\partial E$ are contained in $\partial_- M$, a 2-punctured disk, then, due to the fact that no embedded 2-sphere in $M$ can meet $T$ an odd number of times, $\partial D$ is isotopic to $\partial E$ in $\partial_- M$ and each of $\partial D$ and $\partial E$ bound a once punctured sub disk of $\partial_- M$. Let $\Sigma$ be the 2-punctured sphere that is the union of $D$, $E$ and the annulus in $\partial_- M$ with boundary $\partial D \cup \partial E$. Since $D$ is not isotopic to $E$ then the 3-ball bounded by $\Sigma$ must meet $L$ in a knotted arc and $\Sigma$ is an essential 2-punctured sphere in $M$. Without loss of generality, assume that $\partial D$ is contained in the 1-punctured disk bounded by $\partial E$ in $\partial_- M$. Hence, $E$ is a a once-punctured splitting disk in $M^{D_1}_{T_2 \# T_3 \# T_4}$. By the Splitting Disk Lemma, one of the prime factors of $T_2 \# T_3 \# T_4$ must be a split link with the first strand unknotted.  This contradicts the definition of the $T_i$.  We conclude that $D$ is isotopic to $E$.

Thus, $D$ and $E$, and hence $D_1$ and $E_1$, are isotopic.  By definition of  $D_1$ and $E_1$, this shows that both $T_1$ and $T_1'$ agree up to stacking with a pure braid and that $T_2\# T_3 \# T_4$ and $T_2'\# T_3' \# T_4'$ agree up to stacking with a pure braid.  This completes Step 1.

\textbf{Step 2:}
We now want to show that $T_2=T_2'$ modulo braid equivalence.  We know that $T_2\# T_3 \# T_4$ and $T_2'\# T_3' \# T_4'$ agree up to a pure braid. We apply to this string link ($T_2\# T_3\# T_4$) the same argument as in Step 1, but with the roles of the first and second strands reversed.  This shows that $T_2 =T_2'$ modulo braid equivalence and also that $T_3 \# T_4 = T_3'\# T_4'$ modulo braid equivalence.  This completes Step 2.

\textbf{Step 3:}
We now know that $T_3\# T_4 = T_3' \# T_4'$ modulo braid equivalence so it suffices to consider the case where $L=T_3 \# T_4 = T_3' \# T_4'$ is a 2-string link in $M$ and where each of $T_3\# T_4$ and $T_3' \# T_4'$ correspond to prime factorizations where none of the factors are split links.

Let $D$ be the 2-punctured disk decomposing $L$ into $T_3$ and $T_4$, and let $E$ be the 2-punctured disk decomposing $L$ into $T_3'$ and $T_4'$.
We may assume $\bdd D \cap \bdd E= \emptyset$.  Subject to this constraint, isotope $D$ and $E$ so as to minimize the number of circles in $D\cap E$.

Suppose that this number is nonzero.
We will eventually deduce a contradiction.
Let $\alpha$ be one of these circles which is innermost in $E$.  Let $E_\alpha$ be the punctured disk in $E$ bounded by $\alpha$, and let $D_\alpha$ be the punctured disk in $D$ bounded by $\alpha$.  Since each strand of $L$ intersects each of $D$ and $E$ once, $D_\alpha$ and $E_\alpha$ have the same number of punctures.  If the number of punctures is zero, then $D_\alpha$ can be pushed across $E_\alpha$ to eliminate $\alpha$, contradicting minimality of $|D\cap E|$.  If the number of punctures is one, then $T_3$ or $T_3'$ has a nontrivial splitting disk, which by the same argument given at the end of Step 1, contradicts the assumption that $L$ has no split factors.  So the number of punctures in $D_\alpha$ (and hence also in $E_\alpha$) is exactly 2.  This implies that $\alpha$ is the only circle in $D\cap E$ which is innermost in $E$.  By a similar argument there is a circle $\beta$ in $D\cap E$ which is the unique innermost circle in $D$.  If $D_\beta$ and $E_\beta$ are the punctured disks bounded by $\beta$ in $D$ and $E$, this similar argument also shows that each of these disks has exactly 2 punctures.

So the circles in $D\cap E$ are concentric in both $D$ and $E$, that is, $D\cap E$ separates each of $D$ and $E$ into an inner 2-punctured disk and $|D\cap E|$ concentric annuli, and the 2 punctures in each of $D$ and $E$ are contained in the innermost punctured disks $D_\beta$ and $E_\alpha$.

Let $A$ be the outermost annulus complementary component of $D\cap E$ in $E$.  By the above conclusion, $A$ has no punctures.  Without loss of generality, $A$ is properly embedded in $M^D_{T_4}$. (Otherwise, reverse the roles of $T_3\# T_4$ and $T_3' \# T_4'$ so that $A$ is an embedded subsurface of $D$ in $M^E_{T'_4}$.)
If $A$ is boundary parallel in $M^D_{T_4}$, then there is an isotopy
of $E$ eliminating the component of $D\cap E$ that is contained in $\partial A$, contradicting the minimality of $|D\cap E|$.  Furthermore, $A$ is incompressible since $E$ is incompressible.  Hence $A$ is an essential annulus in $M^D_{T_4}$.

Let $\mathcal{F}$ be the union of the decomposing disks $F_1,...,F_k$ for the prime decomposition of $T_4$.
\begin{itemize}
\item[\textbf{Claim 1:}]
After an isotopy, we can assume $A \cap \mathcal{F}$ is a collection of circles none of which bounds an unpunctured disk in either $A$ or $\mathcal{F}$.  \emph{Proof}: suppose such a circle $\gamma$ bounds an unpunctured disk in one of $\{A, F_i\}$.  If the disk that $\gamma$ bounds in the other of $\{A, F_i\}$ has at least one puncture, then it must have an even number of punctures, contradicting that each strand of $L$ intersects each of $A$ and $F_i$ only once.  So $\gamma$ bounds unpunctured disks in both $A$ and $F_i$.  But then we can eliminate $\gamma$ by isotopy.
\item[\textbf{Claim 2:}]
No circle $\gamma$ of $A \cap \mathcal{F}$ bounds a 1-punctured disk in a component of $\mathcal{F}$.  \emph{Proof}: If some $\gamma$ does bound a 1-punctured disk, then, by taking $\gamma$ to be innermost, the union of this disk together with the disk which $\gamma$ bounds in $E$ is a punctured immersed sphere (with transverse intersections), which meets $L$ either once (if $\gamma$ is inessential in $A$) or 3 times (if $\gamma$ is essential in $A$); this contradicts the fact that a properly embedded 1-manifold in a 3-ball must intersect an immersed closed surface (with transverse self-intersections) in an even number of points.  (The latter fact can be shown by resolving self-intersections and the Generalized Jordan Curve Theorem \cite{GP}.)

\end{itemize}

Now $\mathcal{F}$ separates $M^D_{T_4}$ into components $M_1,...,M_k$ corresponding to the prime factors $L_1,...,L_k$ in our given decomposition of $T_4$, where each $M_i \cong D^2 \x I$.
Since $A$ has no punctures, Claim 1 above implies that no circle in $A \cap \mathcal{F}$ bounds a disk in $A$.  Thus all the circles of $A \cap \mathcal{F}$ are essential in $A$, and each $M_i$ meets $A$ in a (possibly empty) collection of annuli. Since $A$ is essential in $M^D_{T_4}$, $A$ meets some $M_i$ in an essential annulus $A^*$, i.e., an annulus $A^*$ which is knotted.
If $A^*$ is not a cabling annulus, then  $L_i$ must not be prime, a contradiction.  So $A^*$ is a cabling annulus. By Claims 1 and 2, every circle of $A \cap \mathcal{F}$ bounds a 2-punctured disk in each component of $\mathcal{F}$.
Thus both strands of $L_i$ are contained on the complementary component of $A^*$ in $M_i$ homeomorphic to $D^2\times I$, and $L_i$ is a one-strand cable.  This contradicts the definition of $T_4$, since none of the prime factors in its given decomposition were one-strand cables.  So we may now assume that $D$ and $E$ are disjoint.

Suppose  $D$ and $E$ are not isotopic.  Since $D\cap E = \emptyset$, either $D$ is contained in $M^E_{T'_4}$ or $E$ is contained in $M^D_{T_4}$; without loss of generality, suppose $E$ is contained in $M^D_{T_4}$.  As before, let $\mathcal{F}=\bigcup_i F_i$ be the union of the decomposing disks for the prime factors $L_1,...,L_k$ of $T_4$ in our given decomposition, and let $M_1,...,M_k$ be the components that the $F_i$ separate $M^D_{T_4}$ into.  If $E$ can be isotoped to be disjoint from $\mathcal{F}$, then $E$ is contained in some $M_i$.  Since the $L_i$ are prime, then $E$ must be isotopic to some $F_j$. Hence, the cabling annulus for the portion of $L$ in $M^E_{T'_3}$ meets each of $M_1,...,M_j$ demonstrating that each of $L_1,...,L_j$ is a one-strand cable and contradicting the definition of $T_4$. So $E$ must intersect $\mathcal{F}$.  By the same argument as at the beginning of this Step, after minimizing $|E \cap \mathcal{F}|$ there is a unique circle among those in $E \cap \mathcal{F}$ which is innermost in $E$.  Let $F_i$ be the component of $\mathcal{F}$ which this circle is contained in.  We can now apply the same argument as above with $F_i$ playing the role of $D$, as follows:

We deduce that the circles in $E\cap F_i$ separate $E$ into a 2-punctured disk and concentric (unpunctured) annuli.  We consider the outermost annulus in $E$. As in the argument given above, this annulus implies one of the $L_i$ is a one-strand cable, a contradiction to the definition of $T_4$.  Thus, $F_i$ and $E$ must be isotopic, which completes Step 3.

\textbf{Step 4:}
We now know that $T_i = T_i'$ modulo braid equivalence for all $i=1,...,4$.  We also know that the prime factors in $T_i$ and $T_i'$ agree modulo braid equivalence for $i=1,2,3$ by the prime decomposition theorem for knots.  Thus, it remains to show that the prime factors in $T_4$ and $T_4'$ agree modulo braid equivalence.

Suppose we have two prime decompositions of this link, $K_1\# ... \# K_m = T_4 = L_1\# ... \# L_n$ in $M$.  Let $D_1,...,D_{m-1}$ be the decomposing 2-punctured disks for the $K_i$ and let $E_1,...,E_{n-1}$ be the decomposing 2-punctured disks for the $L_i$.  Here we mean for the punctured disks to be in order, i.e., $D_1$ is the closest $D_i$ to $\bdd_-M$ and $E_1$ is the closest $E_i$ to $\bdd_-M$.

We will show that $D_i$ is isotopic to $E_i$ for every $i$.  This will give us uniqueness of the decomposition.  We start with $i=1$. We may assume that the circles $\bdd D_1$ and $\bdd E_1$ are disjoint.

Suppose that we have minimized the number of circles in $D_1 \cap E_1$ (subject to the constraint $\bdd D_1\cap \bdd E_1 =\emptyset$), and suppose first that this number is nonzero. Let $\alpha$ be a component of $D_1 \cap E_1$ that is innermost in $E_1$.  Let $E_\alpha$ be the possibly punctured disk in $E_1$ bounded by $\alpha$.  If $E_\alpha$ has no punctures, then, as previously argued, there is an isotopy of $E_1$ that eliminates a component of $D_1 \cap E_1$ which is a contradiction to the minimality of $|D_1 \cap E_1|$.

If $E_\alpha$ has just one puncture, then by the minimality of $|D_1 \cap E_1|$ and the same arguments given at the end of Step 1, we find a split prime factor, contradicting the definition of $T_4$.  Hence, by the same arguments as in the previous Steps, the innermost circle $\alpha$ bounds a 2-punctured disk in $E_1$, and all circles of $D_1 \cap E_1$ are concentric in $E_1$.

Let $\beta$ be a component of $D_1 \cap E_1$ which is outermost in $E_1$.  Let $A$ be the annulus bounded by $\beta$ and $\bdd E_1$.  By the above paragraph, $\beta$ encloses both punctures of $E_1$, and $A$ has no punctures.  If $A$ is boundary parallel in $M^{D_1}_{K_1}$ or $M^{D_1}_{K_2\# ... \# K_m}$, then we can eliminate $\beta$ (via an isotopy of $E_1$ that moves $\bdd E_1$ past $\bdd D_1$), contradicting our minimality assumption.  Thus $A$ is knotted, i.e., $A$ is an essential annulus, contained in either $M^{D_1}_{K_1}$ or $M^{D_1}_{K_2\# ... \# K_m}$.  In the first case, the primeness of $K_1$ and the fact that $\beta$ encloses both punctures of $E$ imply (as in Step 3)  that $K_1$ is a 1-strand cable, contradicting the definition of $T_4$.  In the second case, we can argue (again, as in Step 3) that some other $K_i$ is a 1-strand cable, a contradiction.

So we may now assume that $D_1$ and $E_1$ can be isotoped to be disjoint.  Then one of these punctured disks lies between $\bdd_-M$ and the other disk.  Without loss of generality, suppose $E_1$ lies between $\bdd_- M$ and $D_1$.  Then since $E_1$ is not isotopic to $\bdd_-M$ and since $K_1$ is prime, we conclude that $E_1$ must be isotopic to $D_1$.  This implies $K_1$ is braid equivalent to $L_1$ and $K_2\#...\# K_m$ is braid equivalent to $L_2\#...\#L_n$.

We may now repeat this argument for the string links $K_2\#...\# K_m$ and $L_2\#...\#L_n$ to conclude $K_2$ is braid equivalent to $L_2$ and $K_3\#...\# K_m$ is braid equivalent to $L_3\#...\#L_n$, and so on.  If $m<n$, then $E_m,...,E_{n-1}$ are decomposing disks between $D_{m-1}$ and $\bdd_+ M$ but isotopic to neither $D_{m-1}$ nor $\bdd_+ M$.  This contradicts the primeness of $K_m$.  If $n<m$, we reach a contradiction by a similar argument, by the primeness of $L_n$.  Thus $m=n$, and $D_i$ is isotopic to $E_i$ for all $i$.  Thus $K_i$ and $L_i$ agree up to braid equivalence for all $i$.  This completes the proof of the theorem.
\end{proof}

Finally we point out an equivalent restatement of this theorem in terms of the language developed in the proof of Proposition \ref{prop:LL}.  In this alternative formulation, we essentially ``remove" all the pure braids from a 2-string link.

\begin{theorem}[Reformulation of Theorem \ref{2LinkPrimeDecomp}]
Any element of $\LL^0_2$ can be written as a product of primes which is unique up to only reordering the split link and one-strand cable factors.
\qed
\end{theorem}

For deducing a corollary concerning arbitrary 2-string links, it is convenient to have alternative definitions of equivalence and braid-equivalence of string links.

\begin{proposition}[Alternative definitions of equivalence and braid-equivalence of string links]
\label{AltStringLinkEquiv}
(1) String links $T_1$ and $T_2$ are equivalent if and only if there is a diffeomorphism of pairs $(M, T_1) \overset{\cong}{\to} (M, T_2)$ whose restriction to the boundary of $M$ (and all of its derivatives) agrees with the identity.

(2) String links $T_1$ and $T_2$ are braid-equivalent if and only if there is a diffeomorphism of pairs $(M, T_1) \overset{\cong}{\to} (M, T_2)$ whose restriction to $(\partial D^2) \x I$ (and all of its derivatives) agrees with the identity.

(3) String links $T_1$ and $T_2$ are braid-equivalent if and only if there is a diffeomorphism of pairs $(M, T_1) \overset{\cong}{\to} (M, T_2)$ which takes the subsets $\partial_\pm M$ to $\partial_\pm M$.
\end{proposition}
\begin{proof}
Part (1):
``$\Rightarrow$":  Given an isotopy of $M$ as in Definition \ref{StringLinkEquivalence}, take desired diffeomorphism to be the time-1 slice of the isotopy.

``$\Leftarrow$":  By Cerf's Theorem \cite{Cerf}, the space of diffeomorphisms of $S^3$ is connected.  As an easy corollary, so is the space of diffeomorphisms of $D^3$ which agree with the identity on the boundary (to all orders of derivatives).  In fact, this follows by considering the fibration
\[
\mathrm{Diff}(D^3,\partial D^3) \to \mathrm{Diff}(S^3) \to \mathrm{Emb}(D^3, S^3)
\]
given by restricting to a hemisphere of $S^3$.  The base is homotopy equivalent to $SO(3)$, which is connected, while the fiber is the abovementioned space of diffeomorphisms of $D^3$ fixed on the boundary.

Now a diffeomorphism $\phi: (M,T_1) \overset{\cong}{\to} (M, T_2)$ is clealry isotopic to one that is the identity outside of a ball $D^3$ contained in $M$.  Combining this with Cerf's Theorem, we get a path (i.e., a diffeotopy) from $\phi$ to the identity.  Taking this path in reverse
 gives the desired isotopy.

 Part (2):  Recall the notion of braid-equivalence from Proposition \ref{BraidEquivalenceProp} as given by an isotopy of $M$ fixing the ``round boundary" of $M$.  Starting with this definition of braid-equivalence, the proof of the ``$\Rightarrow$" direction is the same as that of the ``$\Rightarrow$" direction in part (1).  For the ``$\Leftarrow$" direction, consider the fibration
 \[
\mathrm{Diff}(D^3,\partial D^3) \to \mathrm{Diff}(D^3, D^2) \to \mathrm{Diff}(D^2)
\]
where the total space is the space of diffeomorphisms of $D^3$ which fix one hemisphere $D^2$ of the boundary $S^2$.  Since the base space is connected (in fact, contractible), the total space is too.  Then consider the fibration
 \[
\mathrm{Diff}(D^3, D^2) \to  \mathrm{Diff}(D^3, S^1 \x I) \to \mathrm{Diff}(D^2)
\]
where the total space is the space of diffeomorphisms of $D^3$ fixing a subset $S^1 \x I$ of the boundary $S^2$.  This shows that the total space is connected.  Hence given a diffeomorphism of $M$ fixed on the ``round boundary," there is a path of such diffeomorphisms connecting it to the identity.  This is a braid-equivalence as established in Proposition \ref{BraidEquivalenceProp}.

Part (3):  Using part (2), the proof of ``$\Rightarrow$" is immediate.  For the proof of ``$\Leftarrow$", we need to just consider one more fibration sequence, namely
\[
\mathrm{Diff}(D^3, S^1 \x I) \to \mathrm{Diff}(D^3, \partial_\pm) \to \mathrm{Diff}(D^2 \sqcup D^2) \cong \mathrm{Diff}(D^2) \x \mathrm{Diff}(D^2)
\]
where the total space is the space of diffeomorphisms of $D^3$ which preserve (but don't necessarily fix) two circles in the boundary $S^2$.  This shows that the total space is connected and that a diffeomorphism as in the statement of Part (3) can be joined by a path of such diffeomorphisms to the identity.
\end{proof}

\begin{corollary}\label{cor:main}
A 2-component string link $L$ has a prime decomposition $L=L_1 \# ... \# L_n$, where each $L_i$ is a prime string link.  Such a decomposition is unique up to reordering the factors in the center and up to multiplication by units (braids).
\end{corollary}

\begin{proof}
If $L$ is pure, the corollary follows immediately from Theorem \ref{2LinkPrimeDecomp}.  So assume that $L$ is not pure. The existence of a prime decomposition again follows from Theorem \ref{Freedman}. Given two prime decompositions for $L$, each is defined (up to braid equivalence of the factors) by a complete set of decomposing disks $\mathcal{E}$ and $\mathcal{D}$ respectively. If $\tau$ is the generator of the 2-strand braids, then $L\# \tau$ is a pure braid. However, $(D^2\times I, L)$ is diffeomorphic to $(D^2\times I, L\# \tau)$ via a diffeomorphism
which takes decomposing disks to decomposing disks. So the image of $\mathcal{E}$ and $\mathcal{D}$ under this diffeomorphism are complete sets of decomposing disks for for $L\# \tau$ that decompose $(D^2 \x I, L\# \tau)$ into
pieces diffeomorphic to those in the original decompositions for $L \# \tau$
By part (3) of Proposition \ref{AltStringLinkEquiv}, the prime factors of $L \# \tau$ determined by these decomposing disks are braid-equivalent to those in the decompositions for $L$.
By Theorem \ref{2LinkPrimeDecomp}, the two decompositions for $L\# \tau$ are related by reordering the factors in the center and multiplication of factors by units.
Thus the two decompositions of $L$ are related via reordering the factors in the center and up to multiplication by units.
\end{proof}

\bibliographystyle{plain}
\bibliography{bib}

\begin{bibdiv}
\begin{biblist}

\bib{B1}{article}{
      author={Budney, Ryan},
       title={Little cubes and long knots},
        date={2007},
        ISSN={0040-9383},
     journal={Topology},
      volume={46},
      number={1},
       pages={1\ndash 27},
         url={http://dx.doi.org/10.1016/j.top.2006.09.001},
      review={\MR{2288724 (2008c:55015)}},
}

\bib{B2}{article}{
      author={Budney, Ryan},
       title={An operad for splicing},
        date={2012},
     journal={J. Topol.},
      volume={5},
      number={4},
       pages={945\ndash 976},
}

\bib{BK}{article}{
      author={Burke, John},
      author={Koytcheff, Robin},
       title={An operad for infection of links},
     journal={arXiv:1311.4217},
}

\bib{Cerf}{book}{
      author={Cerf, J.},
       title={{Sur les diff\'{e}omorphismes de la sph\`{e}re de dimension trois
  ($\Gamma_4=0$)}},
   publisher={Springer-Verlag},
        date={1968},
        note={Lecture Notes in Mathematics, No. 53},
}

\bib{FF}{article}{
      author={Freedman, Benedict},
      author={Freedman, Michael~H.},
       title={Kneser-{H}aken finiteness for bounded {$3$}-manifolds locally
  free groups, and cyclic covers},
        date={1998},
        ISSN={0040-9383},
     journal={Topology},
      volume={37},
      number={1},
       pages={133\ndash 147},
         url={http://dx.doi.org/10.1016/S0040-9383(97)00007-4},
      review={\MR{1480882 (99h:57036)}},
}

\bib{GP}{book}{
      author={Guillemin, V.},
      author={Pollack, A.},
       title={Differential topology},
        date={1974},
}

\bib{HL}{article}{
      author={Habegger, Nathan},
      author={Lin, Xiao-Song},
       title={The classification of links up to link-homotopy},
        date={1990},
     journal={J. Amer. Math. Soc.},
      volume={3},
      number={2},
       pages={389\ndash 419},
}

\bib{K}{article}{
      author={Krebes, David},
       title={Units of the string link monoids},
     journal={arXiv:1304.4684},
}

\bib{R76}{book}{
      author={Rolfsen, Dale},
       title={Knots and links},
   publisher={Publish or Perish, Inc., Berkeley, Calif.},
        date={1976},
        note={Mathematics Lecture Series, No. 7},
      review={\MR{0515288 (58 \#24236)}},
}

\bib{S49}{article}{
      author={Schubert, Horst},
       title={Die eindeutige {Z}erlegbarkeit eines {K}notens in {P}rimknoten},
        date={1949},
     journal={S.-B. Heidelberger Akad. Wiss. Math.-Nat. Kl.},
      volume={1949},
      number={3},
       pages={57\ndash 104},
      review={\MR{0031733 (11,196f)}},
}

\bib{W68}{article}{
      author={Waldhausen, Friedhelm},
       title={On irreducible {$3$}-manifolds which are sufficiently large},
        date={1968},
        ISSN={0003-486X},
     journal={Ann. of Math. (2)},
      volume={87},
       pages={56\ndash 88},
      review={\MR{0224099 (36 \#7146)}},
}

\end{biblist}
\end{bibdiv}

\end{document}